\documentclass{amsart}

\usepackage{amssymb,hyperref,graphicx,enumitem,amsmath,bm,url}
\usepackage[margin=1in]{geometry}
\newcommand{\A}{\mathbf A}
\newcommand{\Z}{\mathbf Z}

\newcommand{\Q}{\mathbf Q}
\newcommand{\C}{\mathbf C}

\newcommand{\disc}{\operatorname{disc}}
\newcommand{\Res}{\operatorname{Res}}
\newcommand{\Gal}{\operatorname{Gal}}
\newcommand{\Stab}{\operatorname{Stab}}
\newcommand{\Sym}{\operatorname{Sym}}
\newcommand{\Jac}{\operatorname{Jac}}

\newcommand{\p}{\mathfrak p}
\renewcommand{\P}{{\mathfrak P}}

\newcommand{\calS}{\mathcal S}
\newcommand{\calT}{\mathcal T}
\newcommand{\calP}{\mathcal P}

\newcommand{\calG}{\mathcal G}
\newcommand{\calC}{{\mathcal C}}
\newcommand{\calE}{{\mathcal E}}

\newcommand{\calX}{{\mathcal X}}

\renewcommand{\epsilon}{\varepsilon}

\renewcommand{\iff}{\Leftrightarrow}
\renewcommand{\to}{\rightarrow}

\newcommand{\Spec}{\operatorname{Spec}}

\newtheorem{thm}{Theorem}[section]
\newtheorem{lem}[thm]{Lemma}
\newtheorem{prop}[thm]{Proposition}
\newtheorem{cor}[thm]{Corollary}
\newtheorem{conj}[thm]{Conjecture}
\theoremstyle{definition}
\newtheorem{alg}[thm]{Algorithm}
\newtheorem{rem}[thm]{Remark}
\numberwithin{equation}{section}

\title[A local-global principle]{\small{A local-global principle in the dynamics of quadratic polynomials}}
\author{David Krumm}
\address{Department of Mathematics \\
Colby College \\
Waterville, ME 04901\\
USA}
\email{dkrumm@colby.edu}
\urladdr{http://personal.colby.edu/\textasciitilde dkrumm/}
\begin{document}
\maketitle
\begin{abstract}
Let $K$ be a number field, $f\in K[x]$ a quadratic polynomial, and $n\in\{1,2,3\}$. We show that if $f$ has a point of period $n$ in every non-archimedean completion of $K$, then $f$ has a point of period $n$ in $K$. For $n\in\{4,5\}$ we show that there exist at most finitely many linear conjugacy classes of quadratic polynomials over $K$ for which this local-global principle fails. By considering a stronger form of this principle, we strengthen global results obtained by Morton and Flynn-Poonen-Schaefer in the case $K=\Q$. More precisely, we show that for every quadratic polynomial $f\in\Q[x]$ there exist infinitely many primes $p$ such that $f$ does not have a point of period 4 in the $p$-adic field $\Q_p$. Conditional on knowing all rational points on a particular curve of genus 11, the same result is proved for points of period 5.
\end{abstract}

\section{Introduction}\label{intro_section}

Let $X$ be a set and $\phi:X\to X$ a map. For every nonnegative integer $n$, we denote by $\phi^n$ the $n$-fold composition of $\phi$ with itself. Thus, $\phi^0$ is the identity map, and $\phi^{n+1}=\phi\circ \phi^n$ for $n\ge 0$. We say that an element $x\in X$ is \textit{periodic} under $\phi$ if there exists a positive integer $n$ such that $\phi^n(x)=x$; in that case, the least such $n$ is called the \textit{period} of $x$. In this article we are interested in a particular property of periodic points in the case that $X$ is a number field and $\phi$ is a polynomial map. Thus, let $K$ be a number field and let $f\in K[x]$ be a nonconstant polynomial, viewed as a map $f:K\to K$. For any field extension $\tilde K$ of $K$ we may also regard $f$ as a map $\tilde K\to\tilde K$ and consider periodic points of $f$ in $\tilde K$.
It is clear that if $f$ has a point of period $n$ in $K$, then it has a point of period $n$ in every extension of $K$; in particular, for every finite place $v$ of $K$, $f$ has a point of period $n$ in the completion $K_v$. One may ask whether the converse holds: if $f$ has a point of period $n$ in every non-archimedean completion of $K$, must it then have a point of period $n$ in $K$? The purpose of this article is to address this question in the case that $f$ is a quadratic polynomial.

By a \textit{prime} of $K$ we mean a maximal ideal of the ring of integers of $K$. If $\p$ is a prime of $K$ dividing a rational prime $p$, there is a corresponding valuation $v_{\p}$ on $K$ extending the usual $p$-adic valuation on $\Q$. We denote by $K_{\p}$ the completion of $K$ with respect to the valuation $v_{\p}$. With this notation we state the following local-global principle, which is the main object of interest in this article.
\begin{align*}\tag*{$\calE(f,n)$}
\textit{If $f$ has a point of period $n$ in $K_{\p}$ for every prime $\p$, then $f$ has a point of period $n$ in $K$.}
\end{align*}

We will focus mainly on deciding whether this statement holds true in the case that $f$ is a quadratic polynomial and $n\le 5$. Our first result, proved in Theorems \ref{period12_lgp} and \ref{period3_lgp}, is that the answer is affirmative if $n$ is at most 3.

\begin{thm}\label{123_main_thm}
Let $K$ be a number field, $f\in K[x]$ a quadratic polynomial, and $n\in\{1,2,3\}$. Then the statement $\calE(f,n)$ holds true.
\end{thm}

Still restricting attention to quadratic polynomials, in the cases $n=4$ and $n=5$ we show that counterexamples to this local-global principle---if indeed any exist over a given field $K$---are extremely rare in a precise sense explained below. We say that a polynomial $g\in K[x]$ is \textit{linearly conjugate} to $f$ over $K$ if there exists a linear polynomial $\ell\in K[x]$ such that $g=\ell^{-1}\circ f\circ\ell$. In that case, it is a simple exercise to show that the statements $\calE(f,n)$ and $\calE(g,n)$ are equivalent. If $f$ is a quadratic polynomial, then there is a unique element $c\in K$ such that $f$ is linearly conjugate to the polynomial $f_c(x):=x^2+c$, and hence $\calE(f,n)$ is equivalent to $\calE(f_c,n)$. We prove in Theorems \ref{period4_lgp} and \ref{period5_lgp} that if $n=4$ or 5, then there exist at most finitely many elements $c\in K$ for which the statement $\mathcal E(f_c, n)$ is false. Thus, we obtain the following.

\begin{thm}\label{45_main_thm}
Let $K$ be a number field and let $n\in\{4,5\}$. Then there exist, up to linear conjugacy, only finitely many quadratic polynomials $f\in K[x]$ for which the statement $\mathcal E(f, n)$ fails to hold.
\end{thm}

Focusing now on the case $K=\Q$, we can improve upon Theorem \ref{45_main_thm} by providing more information about the possible exceptions to the principle $\mathcal E(f, n)$. For $n=4$, we show in Theorem \ref{period4_rational_thm} that there is \textit{no} exception; in fact the premise of $\mathcal E(f, 4)$ is always false, in the following strong sense.

\begin{thm}\label{strong_morton_thm}
For every quadratic polynomial $f\in\Q[x]$ there exist infinitely many primes $p$ such that $f$ does not have a point of period $4$ in $\Q_p$.
\end{thm}

This result builds on and strengthens a theorem of Morton \cite{morton_4cycles} stating that there is no quadratic polynomial over $\Q$ having a rational point of period 4. The corresponding statement for points of period 5 was proved by Flynn-Poonen-Schaefer \cite{flynn-poonen-schaefer}. In Theorem \ref{period5_rational_thm} we extend their result as well, though in this case we cannot entirely rule out the failure of the local-global principle; the obstacle in this case is a problem of determining all rational points on a particular algebraic curve. More precisely, let $\calX\subset\A^3=\Spec\Q[a,b,c]$ be the curve of genus $11$ defined by the equations $r_1(a,b,c)=r_0(a,b,c)=0$, where
\begin{align}\label{X_curve_equations}
\begin{split}
r_1(a,b,c) &= (19a + 17)c^2 + (11a^3 + 18a^2 + 22ab + 19a + 18b - 24)c + a^5 + a^4 +\\
&\hspace{5mm} 4a^3b + 3a^3 + 3a^2b + 11a^2 + 3ab^2 + 6ab + 44a + b^2 + 11b + 36, \\
r_0(a,b,c) &= 9c^3 + (19b + 40)c^2 + (11a^2b + 18ab + 11b^2 + 19b + 28)c + a^4b +\\
&\hspace{5mm} a^3b + 3a^2b^2 + 3a^2b + 2ab^2 + 11ab + b^3 + 3b^2 + 44b + 32.
\end{split}
\end{align}
An extensive search for rational points on $\calX$ has yielded only the two points $(1, 8, -2)$ and $(-2,-1,-4/3)$.
\begin{thm}\label{strong_fps_thm}
Suppose that $\calX$ has no rational point other than the two known points. Then for every quadratic polynomial $f\in\Q[x]$ there exist infinitely many primes $p$ such that $f$ does not have a point of period $5$ in $\Q_p$.
\end{thm}

There is hope of dispensing with the assumption that $\calX$ has only two rational points. Indeed, one can show that the Jacobian of the nonsingular projective model of $\calX$ has a 2-dimensional isogeny factor $J$ for which the group $J(\Q)$ has rank 1; this suggests that an application of Chabauty-Coleman techniques may succeed in determining all rational points on $\calX$. The details of this approach will appear in the article \cite{doyle-krumm-wetherell}, which deals with several problems of this type.
 
Based on the results of \cite{flynn-poonen-schaefer} and \cite{morton_4cycles}, Poonen \cite{poonen_prep} conjectured that if $f\in\Q[x]$ is a quadratic polynomial and $n>3$, then $f$ does not have a rational point of period $n$. In view of Theorems \ref{strong_morton_thm} and \ref{strong_fps_thm}, which extend the results of \cite{flynn-poonen-schaefer} and \cite{morton_4cycles}, we propose the following stronger statement.

\begin{conj}\label{main_conj}
Let $f\in\Q[x]$ be a quadratic polynomial and $n>3$ an integer. There exist infinitely many primes $p$ such that $f$ does not have a point of period $n$ in $\Q_p$.
\end{conj}

In addition to our results for $n=4$ and $n=5$, we provide evidence for Conjecture \ref{main_conj} by showing that for every value of $n$, the conjecture holds for `most' quadratic polynomials. To make this more precise we use the notion of a \textit{thin} set in the sense of Serre; see \cite[\S 9.1]{serre_lectures} for details.

\begin{thm}
Let $n$ be a positive integer. There is a thin subset $I(n)\subseteq\Q$ with the following property: if $f\in\Q[x]$ is a quadratic polynomial linearly conjugate to $f_c$ with $c\not\in I(n)$, then there exist infinitely many primes $p$ such that $f$ does not have a point of period $n$ in $\Q_p$.
\end{thm}

Stated briefly, the central idea of this article is to consider a strong form of the local-global principle $\calE(f,n)$, namely the statement \ref{lgp} defined in \S\ref{lgp_failure_section}, which is a statement about a certain \textit{dynatomic polynomial}. Using the Chebotarev Density Theorem together with some observations on the structure of the Galois group of this dynatomic polynomial, we deduce a series of conditions that would necessarily hold if $\calE^{\ast}(f,n)$ were false. These conditions are then studied by explicit computations with symmetric groups. 

This article is organized as follows: the necessary background material on density and dynatomic polynomials is reviewed in \S\ref{background_section}; the main tools for exploring the failure of $\calE^{\ast}(f,n)$ are developed in \S\ref{lgp_failure_section} and applied to the case of quadratic polynomials in \S\ref{quadratic_section}. Finally, we give a brief discussion of Conjecture \ref{main_conj} in \S\ref{strong_poonen_section}.

\section{Preliminaries on Dirichlet density and dynatomic polynomials}\label{background_section}

\subsection{Dirichlet density}\label{density_section} Let $K$ be a number field and let $\mathcal M_K$ denote the set of all primes of $K$. For any prime $\p\in\mathcal M_K$ we denote by $N(\p)$ the norm of $\p$. If $S$ is a subset of $\mathcal M_K$, the \textit{Dirichlet density} of $S$ is defined to be
\[\delta(S)=\lim_{s\to1^+}\frac{\sum_{\p\in S}N(\p)^{-s}}{-\log(s-1)},\]

provided the limit exists. The following properties of the Dirichlet density function will be implicitly used throughout this article; for proofs of these statements we refer to \cite[\S8.B]{cox}. 

Let $S$ and $T$ be subsets of $\mathcal M_K$. Then the following hold.
\begin{itemize}[leftmargin=5mm]
\item $\delta(\mathcal M_K)=1$.
\item If $\delta(S)$ exists, then $0\le\delta(S)\le 1$. 
\item If $S$ is finite, then $\delta(S)=0$.
\item If $\delta(T)$ exists and $S$ differs from $T$ by only finitely many elements, then $\delta(S)$ exists and is equal to $\delta(T)$.
\item If $S$ and $T$ are disjoint and $\delta(S)$ and $\delta(T)$ exist, then $\delta(S\cup T)$ exists and is equal to $\delta(S)+\delta(T)$.
\end{itemize}

For any polynomial $F\in K[x]$ we define
\begin{equation}\label{root_prime_def}
\calS_F=\{\p\in\mathcal M_K\;\vert \;F \text{\;has\;a\;root\;in\;} K_{\p}\}.
\end{equation}

It will be important for our purposes to have a way of determining the density of any set of the form $\calS_F$. The following result, obtained by generalizing an argument of Berend-Bilu \cite{berend-bilu}, provides a way of doing this.

\begin{thm}\label{bb_density_thm}
Let $F\in K[x]$ be a nonconstant polynomial. Let $L$ be a splitting field for $F$, and set $G=\Gal(L/K)$. Let $F_1,\ldots, F_s$ be irreducible polynomials in $K[x]$ such that $F=F_1\cdots F_s$. For $1\le i\le s$, let $\theta_i\in L$ be a root of $F_i$ and set $H_i=\Gal(L/K(\theta_i))$. Finally, let $U=\bigcup_{i=1}^s H_i$. The Dirichlet density of the set $\calS_F$ exists and is given by the formula
\begin{equation}\label{density_formula}
\delta(\calS_F)=\frac{\left|\bigcup_{\sigma\in G}\sigma U\sigma^{-1}\right|}{|G|}.
\end{equation}
\end{thm}

We now establish a few preliminary results used in the proof of Theorem \ref{bb_density_thm}. Let $L/K$ be an extension of number fields and let $E$ be an intermediate field in this extension. Let $\p$ be a prime of $K$ unramified in $L$ (and therefore unramified in $E$). For any prime $\P$ of $E$ dividing $\p$, we denote by $f_{\P/\p}$ the residue degree of $\P$ over $\p$.

\begin{lem}\label{root_in_completion}
Let $m\in K[x]$ be an irreducible polynomial such that $E=K(\theta)$ for some root $\theta$ of $m$. Then $m$ has a root in $K_{\p}$ if and only if there exists a prime $\P$ of $E$ dividing $\p$ such that $f_{\P/\p}=1$.
\end{lem}
\begin{proof}
We refer to \cite[Chap. II, \S 8]{neukirch} and \cite[Chap. II, \S 3]{janusz} for proofs of the standard results used here. Since $\p$ is unramified in $E$, for every prime $\P$ of $E$ dividing $\p$ we have $[E_{\P}:K_{\p}]=f_{\P/\p}$. Let $m=m_1\cdots m_g$ be a factorization of $m$ into irreducible polynomials in $K_{\p}[x]$. There are exactly $g$ primes of $E$ dividing $\p$, which we denote by $\P_1,\ldots, \P_g$. For $1\le i\le g$, let $\alpha_i$ be a root of $m_i$ in $\bar K_{\p}$. The completions $E_{\P_1},\ldots, E_{\P_g}$ are $K_{\p}$-isomorphic, in some order, to the fields $K_{\p}(\alpha_1),\ldots, K_{\p}(\alpha_g)$. Hence, 
\[\{\deg m_i:1\le i\le g\}=\{[K_{\p}(\alpha_i):K_{\p}]:1\le i\le g\}=\{[E_{\P_i}:K_{\p}]:1\le i\le g\}=\{f_{\P_i/\p}:1\le i\le g\}.\]
It follows that there is a factor $m_i$ of degree 1 if and only if some residue degree $f_{\P_i/\p}$ is equal to 1.
\end{proof}

Assume now that $L/K$ is Galois. Let $\P$ be a prime of $E$ dividing $\p$ and let $\calP$ be a prime of $L$ dividing $\P$. Since $L/E$ is Galois, we may consider the decomposition group $D_{\calP/\P}\le\Gal(L/E)$ and the Frobenius automorphism $\left(\frac{L/E}{\calP}\right)$, which generates the group $D_{\calP/\P}$ and has order $f_{\calP/\P}$. (See \cite[Chap. III, \S\S1-2]{janusz} for further details.)  
\begin{lem}\label{frob_in_galois}
Let $\P$ be a prime of $E$ dividing $\p$ and let $\calP$ be a prime of $L$ dividing $\P$. The Frobenius automorphism $\left(\frac{L/K}{\calP}\right)$ belongs to $\Gal(L/E)$ if and only if $f_{\P/\p}=1$.
\end{lem}
\begin{proof}
Let $\sigma=\left(\frac{L/K}{\calP}\right)$. From the definition of decomposition groups it follows that $D_{\calP/\P}=D_{\calP/\p}\cap\Gal(L/E)$. Since $\sigma$ generates $D_{\calP/\p}$, we have
\[\sigma\in\Gal(L/E)\iff \sigma\in D_{\calP/\P}\iff D_{\calP/\P}=D_{\calP/\p}\iff |D_{\calP/\p}:D_{\calP/\P}|=1\iff f_{\P/\p}=1.\qedhere\]
\end{proof}

We use the Artin symbol $\left[\frac{L/K}{\p}\right]$ to denote the conjugacy class of $\Gal(L/K)$ consisting of the Frobenius automorphisms $\left(\frac{L/K}{\calP}\right)$, where $\calP$ ranges over all primes of $L$ dividing $\p$.

\begin{prop}\label{bb_density_prop}
With notation as in Theorem \ref{bb_density_thm}, let $\p$ be a prime of $K$ unramified in $L$. Then $F$ has a root in $K_{\p}$ if and only if $\left[\frac{L/K}{\p}\right]\cap U\ne\emptyset$. 
\end{prop}
\begin{proof}
Suppose that $F$ has a root in $K_{\p}$. Then there is an index $1\le i\le s$ such that the polynomial $F_i$ has a root in $K_{\p}$. Let $E$ be the field $K(\theta_i)$, so that $H_i=\Gal(L/E)$. It follows from Lemma \ref{root_in_completion} that there is a prime $\P$ of $E$ dividing $\p$ such that $f_{\P/\p}=1$. Let $\calP$ be a prime of $L$ dividing $\P$. Applying Lemma \ref{frob_in_galois} we obtain that $\left(\frac{L/K}{\calP}\right)\in H_i$ and therefore $\left[\frac{L/K}{\p}\right]\cap U\ne\emptyset$. This proves one direction of the proposition.

Conversely, suppose that $\left[\frac{L/K}{\p}\right]\cap U\ne\emptyset$. Then there exist a prime $\calP$ of $L$ dividing $\p$ and an index $1\le i\le s$ such that $\left(\frac{L/K}{\calP}\right)\in H_i$. Let $E=K(\theta_i)$ and let $\P$ be the prime of $E$ contained in $\calP$. By Lemma \ref{frob_in_galois} we have $f_{\P/\p}=1$. Applying Lemma \ref{root_in_completion} to the polynomial $m=F_i$ we conclude that $F_i$ has a root in $K_{\p}$. Hence, $F$ has a root in $K_{\p}$.
\end{proof}
 
\begin{proof}[Proof of Theorem \ref{bb_density_thm}]
Let $\calT_F$ be the set of primes $\p\in\mathcal M_K$ such that $\p$ is unramified in $L$ and $F$ has a root in $K_{\p}$. Since $\calS_F$ and $\calT_F$ differ by only finitely many elements, it suffices to show that the right-hand side of \eqref{density_formula} represents the density of $\calT_F$. For every conjugacy class $\calC$ of $G$, let $P_{\calC}$ be the set of primes $\p$ of $K$ that are unramified in $L$ and satisfy $\left[\frac{L/K}{\p}\right]=\calC$. By the Chebotarev Density Theorem we have $\delta(P_{\calC})=|\calC|/|G|$. (See \cite{lagarias-odlyzko} or \cite[Chap. VII, \S 13]{neukirch}.) It follows from Proposition \ref{bb_density_prop} that $\calT_F$ may be written as the disjoint union of the sets $P_{\calC}$, where $\calC$ ranges over all conjugacy classes such that $\calC\cap U\ne\emptyset$. Thus,
\[\delta(\calT_F)=\sum_{\calC\cap U\ne\emptyset}\delta(P_{\calC})=\frac{1}{|G|}\sum_{\calC\cap U\ne\emptyset}|\calC|=\frac{1}{|G|}\left|\bigcup_{\calC\cap U\ne\emptyset}\calC\right|=\frac{\left|\bigcup_{\sigma\in G}\sigma U\sigma^{-1}\right|}{|G|}.\]
This completes the proof of the theorem.
\end{proof}

\begin{cor}\label{irred_density_cor}
If $F\in K[x]$ is irreducible and $\deg(F)>1$, then $\delta(\calS_F)<1$.
\end{cor}
\begin{proof}
Let $L$ be a splitting field for $F$, and set $G=\Gal(L/K)$. Let $\theta\in L$ be a root of $F$ and set $H=\Gal(L/K(\theta))$. In the notation of Theorem \ref{bb_density_thm} we have $U=H$. Since $[K(\theta):K]=\deg(F)>1$, the field $K(\theta)$ is strictly larger than $K$, so $H$ is a proper subgroup of $G$. It is a simple exercise in group theory to show that a finite group cannot be equal to the union of the conjugates of a proper subgroup. Hence, the union of all conjugates of $H$ is strictly smaller than $G$. The result now follows from the formula \eqref{density_formula}.
\end{proof}

\begin{cor}\label{normal_poly_cor}
Let $F\in K[x]$ be irreducible, and suppose that the degree of the splitting field of $F$ over $K$ is equal to the degree of $F$. Then $\delta(\calS_F)=1/\deg(F)$.
\end{cor}
\begin{proof}
This follows from Theorem \ref{bb_density_thm} by noting that in this case $U=\{1\}$ and $|G|=\deg(F)$.
\end{proof}

\begin{rem}\label{bb_code}
In certain cases, the formula \eqref{density_formula} can be used to explicitly compute the value of the density $\delta(\calS_F)$ for a given polynomial $F$. In particular, if $K=\Q$ and the degree of the splitting field $L$ is not too large, then all the data involved in this formula (most importantly the Galois group of $L/\Q$) can be obtained using computer algebra software. Code for carrying out this density computation is available in \cite{paper_code}.
\end{rem}

\subsection{Dynatomic polynomials}\label{dynatomic_section}

The study of periodic points under the action of a polynomial leads naturally to the notion of a dynatomic polynomial. In this section we recall the definition and basic properties of these polynomials; further details on this construction may be found in \cite[\S2]{morton-patel}.

Let $K$ be a field of characteristic $0$ and let $f\in K[x]$ be a nonconstant polynomial. For every positive integer $n$, the \textit{$n$-th dynatomic polynomial} of $f$ is defined by the formula
\[\Phi_{n,f}(x)=\prod_{d|n}\left(f^d(x)-x\right)^{\mu(n/d)},\]
where $\mu$ is the M\"{o}bius function. This formula \textit{a priori} defines only a rational function in $K(x)$, but in fact $\Phi_{n,f}$ is a polynomial in $K[x]$; see Theorem 2.5 in \cite{morton-patel}. Note that the degree of $\Phi_{n,f}$ is given by
\begin{equation}\label{dynatomic_degree}
\deg\Phi_{n,f}=\sum_{d|n}\mu(n/d)(\deg f)^d.
\end{equation}

\begin{prop}[Properties of $\Phi_{n,f}$]\label{dynatomic_properties} Let $\tilde K$ be a field extension of $K$ and let $\alpha\in\tilde K$.
\begin{enumerate}
\item If $\alpha$ has period $n$ under $f$, then $\Phi_{n,f}(\alpha)=0$.
\item If $\alpha$ is a root of $\Phi_{n,f}$, then $f^n(\alpha)=\alpha$.
\item If $\alpha$ is a root of $\Phi_{n,f}$, then $f(\alpha)$ is also a root of $\Phi_{n,f}$.
\item Suppose that $\Phi_{n,f}$ has nonzero discriminant. Then every root of $\Phi_{n,f}$ in $\tilde K$ has period $n$ under $f$.
\end{enumerate}
\end{prop}
\begin{proof}
The first statement follows immediately from the definition of $\Phi_{n,f}$. The second statement follows from the factorization
\begin{equation}\label{dynatomic_factorization}
f^n(x)-x=\prod_{d|n}\Phi_{d,f}(x)
\end{equation}
proved in \cite[Thm. 2.4(a)]{morton-patel}. The third statement follows from the fact that $\Phi_{n,f}(x)$ divides $\Phi_{n,f}(f(x))$; see \cite[Thm. 3.3]{morton-patel}. The fourth statement is a consequence of Theorem 2.4(c) in \cite{morton-patel}, where it is shown that if $\alpha$ is a root of $\Phi_{n,f}$ having period smaller than $n$, then $\alpha$ must be a repeated root of $\Phi_{n,f}$.
\end{proof}

\begin{cor}\label{period-dynatomic_cor}
With notation as in Proposition \ref{dynatomic_properties}, suppose that $\Phi_{n,f}$ has nonzero discriminant. Then $f$ has a point of period $n$ in $\tilde K$ if and only if $\Phi_{n,f}$ has a root in $\tilde K$.
\end{cor}
\begin{proof}
This follows immediately from properties (1) and (4) in the proposition.
\end{proof}

Let $\bar K$ be an algebraic closure of $K$ and let $R\subset\bar K$ be the set of roots of $\Phi_{n,f}$. Proposition \ref{dynatomic_properties} implies that $f$ is a map $R\to R$ and that every element of $R$ is periodic under $f$. We mention two important consequences of these facts:

\begin{itemize}[leftmargin=3mm, itemsep=2mm]
\item The map $f$ is a bijection $R\to R$. Indeed, $f$ is surjective because for every $\alpha\in R$ we may write $\alpha=f(\beta)$ where $\beta=f^{n-1}(\alpha)\in R$. Since $R$ is a finite set, $f$ must also be injective.
\item The set $R$ can be partitioned into orbits under the action of $f$, where the orbit of an element $\alpha\in R$ is the set $\{f^k(\alpha):k\ge 0\}$. We will henceforth call these orbits the \textit{cycles} of $f$, with the understanding that they depend on the value of $n$.
\end{itemize}

We end this section by discussing the effect that linear conjugation of $f$ has on its dynatomic polynomials.
\begin{prop}\label{dynatomic_conjugate_prop}
Let $\ell(x)=ax+b\in K[x]$ be a linear polynomial, and let $g=\ell^{-1}\circ f\circ\ell$. Then
\[\Phi_{n,f}(\ell(x))=a^{\delta_{1n}}\cdot \Phi_{n,g}(x),\]
where $\delta$ is the Kronecker delta function.
\end{prop}
\begin{proof}
By elementary algebra one can verify the identity
\[f^d(\ell(x))-\ell(x)=a\cdot(g^d(x)-x)\] 
for every positive integer $d$. Using this relation we obtain
\begin{align*}
\Phi_{n,f}(\ell(x))&=\prod_{d|n}\left(f^d(\ell(x))-\ell(x)\right)^{\mu(n/d)}\\
&=\prod_{d|n}a^{\mu(n/d)}\left(g^d(x)-x\right)^{\mu(n/d)}\\
&=a^{\sum_{d|n}\mu(n/d)}\cdot\prod_{d|n}\left(g^d(x)-x\right)^{\mu(n/d)}\\
&=a^{\delta_{1n}}\cdot \Phi_{n,g}(x).\qedhere
\end{align*}
\end{proof}

For our purposes in later sections, an important consequence of the above proposition is that the dynatomic polynomials of linearly conjugate maps factor in the same way over every extension field of $K$.

\begin{cor}\label{dynatomic_conjugate_cor}
Let $f, g\in K[x]$ be linearly conjugate over $K$, and let $n$ be a positive integer. Let $\tilde K$ be an extension of $K$, and let
\[\Phi_{n,f}=\alpha\cdot P_1^{e_1}\cdots P_s^{e_s}\]
be a factorization of $\Phi_{n,f}$, where $\alpha\in\tilde K$; $P_1,\ldots, P_s$ are pairwise nonassociate irreducible polynomials in $\tilde K[x]$; and $e_1,\ldots, e_s$ are positive integers. Then there is a factorization
\[\Phi_{n,g}=\beta\cdot Q_1^{e_1}\cdots Q_s^{e_s},\] where $\beta\in\tilde K$; $Q_1,\ldots, Q_s\in\tilde K[x]$ are irreducible and pairwise nonassociate; and $\deg Q_i=\deg P_i$ for every $i$. 
\end{cor}
\begin{proof}
Let $\ell(x)=ax+b$ be a linear polynomial such that $g=\ell^{-1}\circ f\circ\ell$. The result follows from Proposition \ref{dynatomic_conjugate_prop} by letting $\beta=\alpha\cdot a^{-\delta_{1n}}$ and $Q_i(x)=P_i(\ell(x))$ for every $i$.
\end{proof}

\begin{cor}\label{dynatomic_conjugate_cor2}
Let $f, g$, and $n$ be as in Corollary \ref{dynatomic_conjugate_cor}. Then the following hold:
\begin{enumerate}
\item $\Phi_{n,f}$ has a root in $K$ if and only if $\Phi_{n,g}$ has a root in $K$.
\item $\Phi_{n,f}$ has discriminant $0$ if and only if $\Phi_{n,g}$ has discriminant $0$.
\item Suppose that $K$ is a number field. Then the sets $\calS_{\Phi_{n,f}}$ and $\calS_{\Phi_{n,g}}$, defined in \eqref{root_prime_def}, are equal.
\end{enumerate}
\end{cor}
\begin{proof}
Properties (1) and (2) follow from Corollary \ref{dynatomic_conjugate_cor} by considering the factorizations of $\Phi_{n,f}$ and $\Phi_{n,g}$ over $K$. Property (3) follows by considering the factorizations over $K_{\p}$ for every prime $\p$ of $K$.
\end{proof}

Having discussed all of the necessary background material, we now proceed to the main section of the article.

\section{Necessary conditions for failure of the local-global principle}\label{lgp_failure_section}

Let $K$ be a number field, $f\in K[x]$ a nonconstant polynomial, and $n$ a positive integer. We are concerned in this article with the following local-global principle:
\begin{align*}\label{lgp0}\tag*{$\calE(f,n)$}
\textit{If $f$ has a point of period $n$ in $K_{\p}$ for every prime $\p$, then $f$ has a point of period $n$ in $K$.}
\end{align*}

Our approach to studying this principle is to first consider a somewhat different statement, to which we shift attention throughout this section:
\begin{align*}\label{lgp}\tag*{$\mathcal E^{\ast}(f,n)$}
  &\textit{If $\Phi_{n,f}$ does not have a root in $K$, then $\delta(\calS_{\Phi_{n,f}})<1$.}
\end{align*}
Unwrapping the definitions, the conclusion of this statement means that the set of primes $\p$ such that $\Phi_{n,f}$ has a root in $K_{\p}$ has Dirichlet density less than 1. We begin by showing that, for the `typical' polynomial $f$, this new local-global principle is stronger than $\calE(f,n)$. 
\begin{lem}\label{stars_lem}
Suppose that $\Phi_{n,f}$ has nonzero discriminant. Then $\mathcal E^{\ast}(f,n)$ implies $\calE(f,n)$.
\end{lem}
\begin{proof}
Assuming that $\mathcal E^{\ast}(f,n)$ holds and that, for every prime $\p$ of $K$, $f$ has a point of period $n$ in $K_{\p}$, we must show that $f$ has a point of period $n$ in $K$. The proof will be by contradiction; thus, we suppose that $f$ does not have a point of period $n$ in $K$. By Corollary \ref{period-dynatomic_cor}, this implies that $\Phi_{n,f}$ does not have a root in $K$. We may then apply $\mathcal E^{\ast}(f,n)$ to conclude that $\delta(\calS_{\Phi_{n,f}})<1$. In particular, there are infinitely many primes $\p$ lying outside the set $\calS_{\Phi_{n,f}}$. Fix any such prime $\p$. By definition of $\calS_{\Phi_{n,f}}$, the polynomial $\Phi_{n,f}$ does not have a root in $K_{\p}$. Hence, by Corollary \ref{period-dynatomic_cor}, $f$ does not have a point of period $n$ in $K_{\p}$. This contradicts one of the hypotheses and therefore proves the lemma.
\end{proof}

Given that $\mathcal E^{\ast}(f,n)$ typically implies $\mathcal E(f,n)$, it becomes of interest to know under what conditions the former statement might hold. Our goal in this section is to explore the consequences of assuming that $\mathcal E^{\ast}(f,n)$ does \textit{not} hold, with the intent of later proving -- by contradiction -- that $\mathcal E^{\ast}(f,n)$ \textit{does} hold under some assumptions on $n$ and $f$; that will be the general strategy for obtaining the main results of this article.
 
\textbf{Notation.} Let $\bar K$ be an algebraic closure of $K$. Let $L\subset\bar K$ be the splitting field of $\Phi_{n,f}$ and $R\subset L$ the set of roots of $\Phi_{n,f}$. Let $G=\Gal(L/K)$ be the Galois group of $\Phi_{n,f}$. Since every element $\sigma\in G$ maps $R$ to itself, the group $G$ has a natural action on the set $R$. For $\alpha\in R$ we will denote by $G_{\alpha}$ the stabilizer of $\alpha$ in $G$; equivalently,
\begin{equation}\label{G_stabilizer}
G_{\alpha}=\Gal(L/K(\alpha)).
\end{equation}
As noted following Corollary \ref{period-dynatomic_cor}, the set $R$ can be partitioned into cycles; we denote by $r$ the number of cycles in this partition.

\begin{lem}\label{stabilizer_lem}
If $\alpha, \beta\in R$ belong to the same cycle, then $K(\alpha)=K(\beta)$ and $G_{\alpha}=G_{\beta}$.
\end{lem}
\begin{proof}
Since $\alpha$ and $\beta$ are in the same orbit under the action of $f$, there are nonnegative integers $k$ and $j$ such that $\beta=f^k(\alpha)$ and $\alpha=f^j(\beta)$. The first relation implies that $\beta\in K(\alpha)$, and the second that $\alpha\in K(\beta)$. Thus, $K(\alpha)=K(\beta)$. By \eqref{G_stabilizer}, this implies that $G_{\alpha}=G_{\beta}$.
\end{proof}

\begin{lem}\label{irred_factor_lem}
Let $\eta_1,\ldots,\eta_r$ be representatives of the distinct cycles in $R$. Then the set of degrees of irreducible factors of $\Phi_{n,f}$ in $K[x]$ is equal to the set of indices $\{|G: G_{\eta_i}|:1\le i\le r\}$.
\end{lem}
\begin{proof}
Let $P\in K[x]$ be an irreducible factor of $\Phi_{n,f}$. We must show that $\deg P=|G: G_{\eta_i}|$ for some $i$. Let $\alpha\in \bar K$ be a root of $P$. Since $\alpha$ is also a root of $\Phi_{n,f}$, there is some index $i$ such that $\alpha$ and $\eta_i$ belong to the same cycle in $R$. By Lemma \ref{stabilizer_lem} we have $G_{\alpha}=G_{\eta_i}$ and therefore $|G:G_{\alpha}|=|G:G_{\eta_i}|$. By \eqref{G_stabilizer}, the index $|G:G_{\alpha}|$ is equal to the degree of the extension $K(\alpha)/K$, which is the degree of $P$. Hence $\deg P=|G:G_{\eta_i}|$, as required. Now, fixing $i\in \{1,\ldots, r\}$ we have to show that the index $|G:G_{\eta_i}|$ is equal to the degree of some irreducible factor of $\Phi_{n,f}$. Let $P$ be the minimal polynomial of $\eta_i$ over $K$. Note that $P$ divides $\Phi_{n,f}$ since $\eta_i$ is a root of $\Phi_{n,f}$. Moreover, we have $\deg P=[K(\eta_i):K]=|G:G_{\eta_i}|$ by \eqref{G_stabilizer}. This completes the proof.
\end{proof}

\begin{prop}\label{density1_prop}
Suppose that $\Phi_{n,f}$ has no root in $K$ and that $\delta(\calS_{\Phi_{n,f}})=1$. Let $\eta_1,\ldots,\eta_r$ be representatives of the distinct cycles in $R$. Then the following hold:
\begin{enumerate}
\item $G=\bigcup_{i=1}^rG_{\eta_i}$.
\item For $1\le i\le r$, $G \ne G_{\eta_i}$.
\end{enumerate}
\end{prop}
\begin{proof}
Write $\Phi_{n,f}=P_1\cdots P_s$, where the polynomials $P_j\in K[x]$ are irreducible. For every index $1\le j\le s$, let $\theta_j\in R$ be a root of the polynomial $P_j$ and set $H_j=\Gal(L/K(\theta_j))$. Note that $H_j=G_{\theta_j}$ by \eqref{G_stabilizer}. Since the set $\calS_{\Phi_{n,f}}$ has density $1$, Theorem \ref{bb_density_thm} implies that
\begin{equation}\label{G_conjugate_union}
G=\bigcup_{j=1}^s\bigcup_{\sigma\in G}\sigma H_j\sigma^{-1}=\bigcup_{j=1}^s\bigcup_{\sigma\in G}\sigma G_{\theta_j}\sigma^{-1}=\bigcup_{j=1}^s\bigcup_{\sigma\in G}G_{\sigma(\theta_j)}.
\end{equation}
Fix $\sigma\in G$ and $j\in\{1,\ldots, s\}$. Since $\sigma(\theta_j)\in R$, there is some index $i\in\{1,\ldots, r\}$ such that $\sigma(\theta_j)$ belongs to the same cycle as $\eta_i$. By Lemma \ref{stabilizer_lem} we then have $G_{\sigma(\theta_j)}=G_{\eta_i}$ and in particular $G_{\sigma(\theta_j)}\subseteq\bigcup_{i=1}^rG_{\eta_i}$. Since this containment holds for every $\sigma$ and $j$, it follows from \eqref{G_conjugate_union} that $G=\bigcup_{i=1}^rG_{\eta_i}$. This proves (1).

To see (2), fix any $i\in \{1,\ldots, r\}$. By Lemma \ref{irred_factor_lem}, the index $|G:G_{\eta_i}|$ is equal to the degree of some irreducible factor of $\Phi_{n,f}$. By assumption, $\Phi_{n,f}$ has no root in $K$ and therefore no factor of degree $1$. Hence, $|G:G_{\eta_i}|\ne 1$ and therefore $G\ne G_{\eta_i}$, as claimed.
\end{proof}

Assuming that the statement $\mathcal E^{\ast}(f,n)$ does not hold, the above proposition imposes a considerable restriction on the group $G$; for instance, it implies that $G$ is a union of $r$ proper subgroups. In order to exploit this property and make it amenable to computation, we will now embed $G$ into a symmetric group.

Let $\Sym(R)$ be the symmetric group acting on the set $R$. The group $G$ acts faithfully on $R$ (since $L$ is generated over $K$ by the elements of $R$), so the permutation representation $G\to\Sym(R)$ is injective; we will henceforth identify $G$ with its image under this embedding. More concretely, we view elements of $G$ as bijections $R\to R$ by restricting them to $R$. As noted following Corollary \ref{period-dynatomic_cor}, $f$ is a bijection $R\to R$; thus, we may also regard $f$ as an element of $\Sym(R)$.

Let $N=\# R$ and let $S_N$ be the symmetric group acting on the set $\{1,2,\ldots, N\}$. By labeling the elements of $R$, or more precisely by fixing a bijection $\ell:\{1,2,\ldots, N\}\to R$, we may identify $\Sym(R)$ with $S_N$ and therefore view $G$ as a subgroup of $S_N$. In what follows we will take care to choose the labeling so that it is compatible with the action of $f$ on $R$.

We assume now that the discriminant of $\Phi_{n,f}$ is nonzero. By part (4) of Proposition \ref{dynatomic_properties}, this implies that every element of $R$ has period $n$ under $f$, so each of the $r$ cycles in $R$ contains $n$ roots of $\Phi_{n,f}$. Hence,
\[N=\#R=nr=\deg\Phi_{n,f}.\]
Let $\eta_1,\ldots,\eta_r$ be representatives of the distinct cycles in $R$. We define a map $\ell:\{1,2,\ldots, N\}\to R$ by
\[\ell(ni-j)=f^{n-j}(\eta_i)\;\;\text{for}\;\; 1\le i\le r\;\;\text{and}\;\;0\le j<n.\]
It is a simple calculation to check that $\ell$ is a well-defined bijection. The map $\ell$ can be regarded as a labeling of the elements of $R$ as in the figure below, where vertices represent the elements of $R$ and arrows represent the action of $f$.
\begin{figure}[h]
\includegraphics[scale=0.27]{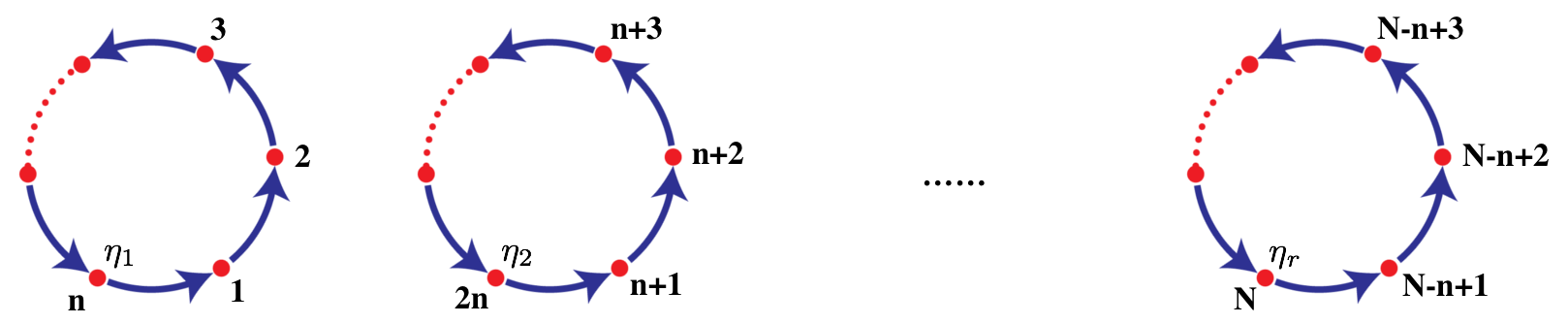}
\caption{Labeling of the roots of $\Phi_{n,f}$.}
\label{roots_labeling}
\end{figure}

Let $\iota:\Sym(R)\to S_N$ be the isomorphism defined by
\begin{equation}\label{iota_def}
\iota(p)=\ell^{-1}\circ p\circ\ell,
\end{equation} 

and let $\pi\in S_N$ be the permutation with the following cycle decomposition:
\begin{equation}\label{pi_def}
\pi=(1,2,\ldots, n)(n+1,n+2,\ldots, 2n)\cdots(N-n+1, N-n+2, \ldots, N).
\end{equation} 

\begin{lem}\label{ell_f_lem}
With notation as above, we have $f\circ\ell=\ell\circ\pi$. Hence, $\iota(f)=\pi$.
\end{lem}
\begin{proof}
Let $k\in\{1,2,\ldots, N\}$; we must show that $f(\ell(k))=\ell(\pi(k))$. Suppose first that $n|k$, say $k=ni$ for some $1\le i\le r$. Then $\pi(k)=k-n+1=ni-(n-1)$, so $\ell(\pi(k))=f^{n-(n-1)}(\eta_i)=f(\eta_i)$. On the other hand, $\ell(k)=\ell(ni)=f^{n}(\eta_i)=\eta_i$, so $f(\ell(k))=f(\eta_i)$. Thus, we have $\ell(\pi(k))=f(\eta_i)=f(\ell(k))$, as required. Now suppose that $n\nmid k$ and write $k=ni-j$ for some $1\le i\le r$ and $1\le j<n$. Then $f(\ell(k))=f(f^{n-j}(\eta_i))=f^{n-(j-1)}(\eta_i)$. 
Since $n\nmid k$, we have $\pi(k)=k+1=ni-(j-1)$. Given that $0\le j-1<n$, the definition of $\ell$ implies that $\ell(\pi(k))=\ell(ni-(j-1))=f^{n-(j-1)}(\eta_i)$. Therefore, $f(\ell(k))=f^{n-(j-1)}(\eta_i)=\ell(\pi(k))$. This shows that $f\circ\ell=\ell\circ\pi$. The statement that $\iota(f)=\pi$ now follows from the definition of $\iota$.
\end{proof}

\begin{lem}\label{wreath_subgp_lem}
Let $W$ be the centralizer of $\pi$ in $S_N$ and let $\calG=\iota(G)$. Then $\calG$ is a subgroup of $W$.
\end{lem}
\begin{proof}
It suffices to show that $\calG$ is contained in $W$. For every $\alpha\in L$ and every $\sigma\in G$ we have $\sigma\left(f(\alpha)\right)=f\left(\sigma(\alpha)\right)$; in particular, this holds for every $\alpha\in R$. Thus, as elements of the group $\Sym(R)$, $f$ and $\sigma$ commute. It follows that $G$ is contained in the centralizer of $f$ in $\Sym(R)$. Applying the map $\iota$ we see that the group $\calG$ is contained in the centralizer of $\iota(f)$ in $S_N$, which is $W$, by Lemma \ref{ell_f_lem}. Thus, $\calG$ is contained in $W$.
\end{proof}
\begin{rem}
The above lemma is essentially a restatement of the known fact that the Galois group of $\Phi_{n,f}$ is isomorphic to a subgroup of the wreath product $(\Z/n\Z)\wr S_r$. (See, for instance, Theorem 3.56 in \cite[p. 125]{silverman_dynamics}.) Indeed, the group $W$ is a concrete realization of this wreath product as a subgroup of $S_N$, and the group $\calG$ is a subgroup of $W$ isomorphic to $G$.
\end{rem}

By combining the restrictions on $G$ provided by Proposition \ref{density1_prop} with the embedding of $G$ in $S_N$ as a subgroup of $W$, we obtain the following theorem, which is the main result of this section.

\textbf{Notation.} For any subgroup $H$ of $S_N$ and any element $k\in\{1,2,\ldots, N\}$, we denote by $\Stab_H(k)$ the stabilizer of $k$ in $H$.

\begin{thm}\label{density1_thm}
Suppose that the discriminant of $\Phi_{n,f}$ is nonzero, $\Phi_{n,f}$ has no root in $K$, and $\delta(\calS_{\Phi_{n,f}})=1$. Then the map $\iota$ defined in \eqref{iota_def} restricts to an isomorphism $\iota:G\to\calG$, where $\calG$ is a subgroup of $W$ having the following properties:
\begin{enumerate}
\item $\calG\subseteq\bigcup_{i=1}^r\Stab_W(ni)$.
\item For $1\le i\le r$, $\calG\not\subseteq\Stab_W(ni)$.
\item The set of degrees of irreducible factors of $\Phi_{n,f}$ is equal to the set $\{|\calG: \Stab_{\calG}(ni)|:1\le i\le r\}.$
\end{enumerate}
\end{thm}
\begin{proof}
By Lemma \ref{wreath_subgp_lem}, the group $\calG=\iota(G)$ is a subgroup of $W$, and clearly $\iota:G\to\calG$ is an isomorphism. We will show that $\calG$ has the three stated properties. By definition of the map $\ell$, we have $\ell(ni)=\eta_i$ for $1\le i\le r$. Thus,
\begin{equation}\label{stabilizer_iso}
\iota(G_{\eta_i})=\Stab_{\calG}(\ell^{-1}(\eta_i))=\Stab_{\calG}(ni).
\end{equation}
Since there is exactly one root $\eta_i$ in each cycle in $R$, Proposition \ref{density1_prop} implies that
\[G=\bigcup_{i=1}^rG_{\eta_i}\;\;\text{and}\;\; G\ne G_{\eta_i}\;\;\text{for}\;\; 1\le i\le r.\]
Using \eqref{stabilizer_iso} we obtain
\[\calG=\iota(G)=\bigcup_{i=1}^r\Stab_{\calG}(ni)\;\; \text{and}\;\;\calG\ne\Stab_{\calG}(ni)\;\;\text{for}\;\;1\le i\le r.\] 
Thus, $\calG$ has properties (1) and (2).
Furthermore, for $1\le i\le r$ we have $|\calG: \Stab_{\calG}(ni)|=|G:G_{\eta_i}|$, so it follows from Lemma \ref{irred_factor_lem} that $\calG$ has property (3) as well.
\end{proof}

Assuming that the local-global principle $\mathcal E^{\ast}(f,n)$ fails to hold, Theorem \ref{density1_thm} can be used to obtain a finite list of groups that includes a group isomorphic to $G$. Indeed, it suffices for this purpose to find all the subgroups of $W$ having properties (1) and (2) of the theorem, and this can be done using computer algebra software. In addition to this list of groups, we can find the possible degrees of irreducible factors of $\Phi_{n,f}$ by using property (3); this fact will be essential for our analysis in \S\ref{quadratic_section}. Before further discussing this algorithmic approach, we prove a modified version of Theorem \ref{density1_thm} that will make the computations significantly more efficient: we show that instead of checking properties (1) and (2) for \textit{all} the subgroups of $W$, it suffices to consider one subgroup from each conjugacy class of subgroups.

\begin{thm}\label{density1_conj_thm}
Suppose that the discriminant of $\Phi_{n,f}$ is nonzero, $\Phi_{n,f}$ has no root in $K$, and $\delta(\calS_{\Phi_{n,f}})=1$. Let $H_1,\ldots, H_t$ be subgroups of $W$ representing all the conjugacy classes of subgroups of $W$. Then there is some index $m\in\{1,\ldots, t\}$ such that the following hold:
\begin{enumerate}
\item $H_m\subseteq\bigcup_{i=1}^r\Stab_W(ni)$.
\item For $1\le i\le r$, $H_m\not\subseteq\Stab_W(ni)$.
\item The set of degrees of irreducible factors of $\Phi_{n,f}$ is equal to the set $\{|H_m: \Stab_{H_m}(ni)|:1\le i\le r\}$.
\item There exist a bijection $u:\{1,\ldots, N\}\to R$ and an isomorphism $\rho:G\to H_m$ such that $u\circ\pi=f\circ u$ and, for every $g\in G$, $g\circ u=u\circ(\rho(g))$.
\end{enumerate}
\end{thm}
\begin{proof}
By Theorem \ref{density1_thm}, the group $\calG=\iota(G)$ is a subgroup of $W$ having properties (1)-(3) of that theorem. The subgroup $\calG$ must be conjugate to $H_m$ for some index $m$; we claim that $H_m$ has the four properties listed above. To simplify notation, let $H=H_m$. Let $\sigma\in W$ be such that $H=\sigma\calG\sigma^{-1}$, and let $c_{\sigma}$ be the automorphism of $W$ given by $w\mapsto \sigma\circ w\circ\sigma^{-1}$. We begin by verifying property (4). Let $u=\ell\circ\sigma^{-1}$, which is a bijection $\{1,\ldots, N\}\to R$. Note that since $\sigma^{-1}\in W$, then $\sigma^{-1}$ commutes with $\pi$. In addition, by Lemma \ref{ell_f_lem} we have $\ell\circ\pi=f\circ\ell$. Thus,
\[u\circ\pi=\ell\circ\sigma^{-1}\circ\pi=\ell\circ\pi\circ\sigma^{-1}=f\circ\ell\circ\sigma^{-1}=f\circ u.\]

By definition of $\sigma$, the map $c_{\sigma}$ restricts to an isomorphism $\calG\to H$. Since the map $\iota$ is an isomorphism $G\to\calG$, defining $\rho=c_{\sigma}\circ\iota$ we obtain an isomorphism $\rho:G\to H$. For any element $g\in G$, the relation $g\circ u=u\circ(\rho(g))$ follows easily from the definitions. Thus, (4) is satisfied.

The following properties of $c_{\sigma}$ will be useful for checking that (1)-(3) hold:
\begin{itemize}
\item For every element $k\in\{1,\ldots, N\}$, $c_{\sigma}$ maps $\Stab_W(k)$ to $\Stab_W(\sigma(k))$.
\item $c_{\sigma}$ permutes the subgroups in the set $\{\Stab_W(ni):1\le i\le r\}$.
\end{itemize}
We explain the second point. It is easy to see that if $a$ and $b$ are elements of $\{1,\ldots, N\}$ belonging to the same $\pi$-orbit, then $\Stab_W(a)=\Stab_W(b)$. Now fix any index $1\le i\le r$. Since the elements $n,2n,\ldots, rn$ represent the distinct orbits of the action of $\pi$ on $\{1,\ldots, N\}$, the element $\sigma(ni)$ must belong to the same orbit as one of these; say $nj$. We then have $c_{\sigma}\left(\Stab_W(ni)\right)=\Stab_W(\sigma(ni))=\Stab_W(nj)$. This shows that the set $\{\Stab_W(ni):1\le i\le r\}$ is invariant under $c_{\sigma}$, and therefore $c_{\sigma}$ is a permutation of this set.

By Theorem \ref{density1_thm} we have $\calG\subseteq\bigcup_{i=1}^r\Stab_W(ni)$. Applying the map $c_{\sigma}$ we obtain
\[H\subseteq\bigcup_{i=1}^rc_{\sigma}\left(\Stab_W(ni)\right)=\bigcup_{i=1}^r\Stab_W(ni).\]
Thus, $H$ satisfies (1). Property (2) follows similarly from the corresponding property of $\calG$. To show that $H$ satisfies (3) we must prove that
\begin{equation}\label{conj_index}
\{|\calG: \Stab_{\calG}(ni)|:1\le i\le r\}=\{|H: \Stab_{H}(nj)|:1\le j\le r\}.
\end{equation}

Fix any index $1\le i\le r$. A simple calculation shows that $c_{\sigma}:\calG\to H$ maps $\Stab_{\calG}(ni)$ to $\Stab_H(\sigma(ni))$. There is some index $j$ such that $\sigma(ni)$ and $nj$ are in the same $\pi$-orbit, and therefore $\Stab_H(\sigma(ni))=\Stab_H(nj)$. Thus, we have
\[|\calG:\Stab_{\calG}(ni)|=|H:\Stab_H(\sigma(ni))|=|H: \Stab_{H}(nj)|.\]
This proves one containment in \eqref{conj_index}; the reverse containment follows by a similar argument.
\end{proof}

Theorem \ref{density1_conj_thm} provides a theoretical basis for an algorithm that can be used to better understand cases where the local-global principle \ref{lgp} may not hold. We state the algorithm first and then explain its precise relation to this problem.
 
\begin{alg}\label{main_algorithm}\mbox{}

Input: Positive integers $n$ and $s$.\\
Output: A list of pairs $(H, I)$, where $H$ is a group and $I$ is a set of integers.
\begin{enumerate}
\item Let $\displaystyle N=\sum_{d|n}\mu(n/d)s^d$ and $r=N/n$.
\item Construct the permutation $\pi\in S_N$ defined in \eqref{pi_def}.
\item Compute the centralizer $W$ of $\pi$. 
\item Determine subgroups $H_1,\ldots, H_t$ representing all the conjugacy classes of subgroups of $W$.
\item Create an empty list $\calP$.
\item Determine all the groups $H\in\{H_1,\ldots, H_t\}$ satisfying
\begin{itemize}
\item $H\subseteq\bigcup_{i=1}^r\Stab_W(ni)$\; and
\item For $1\le i\le r$, $H\not\subseteq\Stab_W(ni)$.
\end{itemize}
\item For every such group $H$:
\begin{enumerate}
\item Compute the set of indices $I=\{|H: \Stab_{H}(ni)|:1\le i\le r\}$.
\item Include the pair $(H, I)$ in the list $\calP$.
\end{enumerate}
\item Return the list $\calP$.
\end{enumerate}
\end{alg}

There are efficient methods for carrying out steps $3$ and $4$ of Algorithm \ref{main_algorithm}: a method for computing centralizers in permutation groups is discussed in \cite{leon}, and a method for computing representatives of the conjugacy classes of subgroups of a finite group is given in \cite{cannon-cox-holt}. In step 6, one may of course compute all the stabilizers $\Stab_W(ni)$ and check containments by brute force; however, we suggest a more efficient way of checking the two conditions in this step. To test the containment $H\subseteq\bigcup_{i=1}^r\Stab_W(ni)$, it suffices to test, for every $h\in H$, whether $h$ fixes some element $ni$. To test the containment $H\subseteq\Stab_W(ni)$, it suffices to find a set of generators of $H$ and test whether every generator fixes $ni$.

\begin{thm}\label{main_alg_thm}
Let $n$ and $s$ be positive integers. Let $K$ be a number field and $f\in K[x]$ a polynomial of degree $s$. Let $G$ be the Galois group of $\Phi_{n,f}$ and $D$ the set of degrees of irreducible factors of $\Phi_{n,f}$. Let $N=\deg\Phi_{n,f}$ and let $R$ be the set of roots of $\Phi_{n,f}$ in $\bar K$. Suppose that $\disc\Phi_{n,f}\ne 0$ and that the local-global principle \ref{lgp} fails to hold. Letting $\calP$ be the output of Algorithm \ref{main_algorithm} with input $(n,s)$, there is then a pair $(H, I)$ in $\calP$ such that $D=I$ and the following holds: there exist a bijection $u:\{1,\ldots, N\}\to R$ and an isomorphism $\rho:G\to H$ such that $u\circ\pi=f\circ u$ and, for every $g\in G$, $g\circ u=u\circ(\rho(g))$.
\end{thm}
\begin{proof}
Note that all of the constructions carried out and the results obtained in this section up to and including Theorem \ref{density1_conj_thm} apply in this context. We will therefore use here the notation introduced earlier in the section. In particular, the objects $r$, $\ell$, $\iota$, $\pi$, and $W$ are defined as above. The degree formula \eqref{dynatomic_degree} and the relation $N=nr$ imply that $N$ and $r$ are the numbers computed in step 1 of Algorithm \ref{main_algorithm}, and therefore $\pi$ and $W$ are the objects computed in steps 2 and 3. Let $H_1,\ldots, H_t$ be the groups computed in step 4. By Theorem \ref{density1_conj_thm}, there is a group $H\in\{H_1,\ldots, H_t\}$ satisfying conditions (1) and (2) of that theorem, and such that
\begin{equation}\label{alg_proof_iso}
D=\{|H: \Stab_{H}(ni)|:1\le i\le r\}.
\end{equation}
Furthermore, the theorem yields the existence of maps $u$ and $\rho$ with the properties stated above. We claim that the pair $(H,D)$ must belong to the output list $\calP$; this will complete the proof. The conditions (1) and (2) of Theorem \ref{density1_conj_thm} imply that the group $H$ is necessarily found in step 6 of the algorithm. In step 7(a), the set $I$ that is computed is then equal to $D$, by \eqref{alg_proof_iso}. Hence, step 7(b) guarantees that the pair $(H,D)$ is in the list $\calP$.
\end{proof}

Having at this point developed all of the core ideas of this article, we now proceed to apply them to the particular case of quadratic polynomials.

\section{Periodic points of quadratic polynomials}\label{quadratic_section}
The results of \S\ref{lgp_failure_section} apply to any nonconstant polynomial $f\in K[x]$ and any positive integer $n$. In this section we will restrict attention to the case where $f$ is a quadratic polynomial and $n\le 5$; our main goal is to study the statements \ref{lgp} and \ref{lgp0} in this case. The key elements of our approach are Algorithm \ref{main_algorithm} and Theorem \ref{main_alg_thm}. For the purposes of this paper, an implementation of the algorithm in the software system Sage \cite{sage} will be used. The source code of our implementation is available in \cite{paper_code}.

For every element $c\in K$ we define a polynomial $f_c\in K[x]$ by $f_c(x)=x^2+c$. Recall that for every quadratic polynomial $f\in K[x]$ there is a unique element $c\in K$ such that $f$ is linearly conjugate to $f_c$. To ease notation we will write $\Phi_{n,c}$ instead of $\Phi_{n,f_c}$ for the $n$-th dynatomic polynomial of $f_c$.  The polynomials $\Phi_{n,f}$ and $\Phi_{n,c}$ share several properties; in particular, by Corollaries \ref{dynatomic_conjugate_cor} and \ref{dynatomic_conjugate_cor2} we have the following.

Let $\tilde K$ be an extension of $K$. Then:
\begin{itemize}
\item $\Phi_{n,f}$ and $\Phi_{n,c}$ factor in the same way in $\tilde K[x]$.
\item $\Phi_{n,f}$ has a root in $\tilde K$ if and only if $\Phi_{n,c}$ has a root in $\tilde K$.
\item $\disc\Phi_{n,f}=0\iff\disc\Phi_{n,c}=0$.
\item The sets $\calS_{\Phi_{n,f}}$ and $\calS_{\Phi_{n,c}}$ are equal.
\end{itemize}

These basic facts will henceforth be used without explicit mention.

\subsection{The local-global principle for periods 1, 2, and 3}
\begin{thm}\label{period12_lgp} 
Let $f\in K[x]$ be a quadratic polynomial and let $n\in\{1,2\}$. Then the statements \ref{lgp} and \ref{lgp0} hold true.
\end{thm}
\begin{proof}
We begin by proving that $\mathcal E^{\ast}(f,n)$ holds. Assuming that $\Phi_{n,f}$ has no root in $K$, we must show that $\delta(\calS_{\Phi_{n,f}})<1$. By the formula \eqref{dynatomic_degree} we have $\deg\Phi_{n,f}=2$, so the assumption that $\Phi_{n,f}$ has no root implies that it is irreducible. Applying Corollary \ref{normal_poly_cor} we then obtain $\delta(\calS_{\Phi_{n,f}})=1/2<1$, as required.

Now consider the statement $\mathcal E(f,n)$. By Lemma \ref{stars_lem}, if $\disc\Phi_{n,f}\ne 0$, then we already know that $\mathcal E(f,n)$ holds. Thus, it remains only to prove that $\mathcal E(f,n)$ holds in the case where $\disc\Phi_{n,f}=0$. In that case, $\Phi_{n,f}$ is a quadratic polynomial with discriminant $0$, so we may write $\Phi_{n,f}=\lambda(x-\alpha)^2$ for some elements $\lambda, \alpha\in K$. Since $f(\alpha)$ is a root of $\Phi_{n,f}$ (by Proposition \ref{dynatomic_properties}), we must have $f(\alpha)=\alpha$.

\underline{\textit{Case 1:}} $n=1$. Since $\alpha$ is fixed by $f$, then $f$ has a point of period $1$ in $K$, and therefore $\mathcal E(f,1)$ holds.

\underline{\textit{Case 2:}} $n=2$. In this case the premise of $\mathcal E(f,2)$ cannot hold, and therefore $\mathcal E(f,2)$ is true. Indeed, suppose that $f$ has a point $\tilde\alpha$ of period $2$ in an extension of $K$. Then $\tilde\alpha$ is a root of $\Phi_{2,f}$, so $\tilde\alpha=\alpha$, and therefore $\tilde\alpha$ has period $1$; a contradiction. Hence, $f$ does not have a point of period $2$ in any extension of $K$. In particular, there is no prime $\p$ such that $f$ has a point of period $2$ in $K_{\p}$.
\end{proof}

\begin{rem} When $\disc\Phi_{n,f}\ne 0$, the conclusion of Theorem \ref{period12_lgp} can also be reached by applying Algorithm \ref{main_algorithm} and Theorem \ref{main_alg_thm}. Since that will be our strategy for $n=3,4,5$, we illustrate the method by treating the case $n=1$, which can easily be done by hand. Thus, suppose that $f\in K[x]$ is a quadratic polynomial such that  $\disc\Phi_{1,f}\ne 0$ and the statement $\mathcal E^{\ast}(f,1)$ does not hold. We may then apply Theorem \ref{main_alg_thm} with $n=1$ and $s=2$. Following the steps of Algorithm \ref{main_algorithm} with input $(1,2)$, we obtain $N=2$, $r=2$; $\pi$ is the identity permutation in $S_2$; and $W=S_2$. The conjugacy classes of subgroups of $W$ are represented by the groups $\{1\}$ and $S_2$. The stabilizers occurring in step 6 of the algorithm are both trivial, so their union is the trivial subgroup. Hence, in step 6 the two conditions on $H$ cannot both be satisfied. It follows that the output of the algorithm is empty. Now, Theorem \ref{main_alg_thm} implies, in particular, that the output of the algorithm is \textit{nonempty}, so we have reached a contradiction. Therefore, $\mathcal E^{\ast}(f,1)$ must hold.
\end{rem}

\begin{thm}\label{period3_lgp}
Let $f\in K[x]$ be a quadratic polynomial and let $n=3$. Then the statements \ref{lgp} and \ref{lgp0} hold true.
\end{thm}
\begin{proof}
Let $c\in K$ be the unique element such that $f$ is linearly conjugate to $f_c$. The proof will be divided into three cases determined by the discriminant of $\Phi_{3,c}$, which is given by
\[\disc \Phi_{3,c} = -(7+4c)^3(7+4c+16c^2)^2.\]

\underline{\textit{Case 1:}} $(7+4c)(7+4c+16c^2)\ne 0$. In this case we have $\disc\Phi_{3,f}\ne 0$. Hence, by Lemma \ref{stars_lem}, it suffices to prove that $\mathcal E^{\ast}(f,3)$ holds. The proof will be by contradiction; thus, we suppose that $\mathcal E^{\ast}(f,3)$ fails to hold. We may then apply Theorem \ref{main_alg_thm} with $n=3$ and $s=2$. Let $\calP$ be the output of Algorithm \ref{main_algorithm} with input $(3,2)$. Using our implementation of the algorithm (available in \cite{paper_code}), we obtain $\calP=\emptyset$; the steps of the algorithm leading to this result are described below. Note that this fact completes the proof in this case: indeed, Theorem \ref{main_alg_thm} states that there is a pair $(H,I)$ in $\calP$ such that the Galois group of $\Phi_{3,f}$ is isomorphic to $H$. Given that $\calP$ is in fact empty, we arrive at the desired contradiction.

In what follows we use the notation of Algorithm \ref{main_algorithm}. With input $(n,s)=(3,2)$ we obtain
\[N=6, \;r = 2,\;\; \text{and}\;\;\pi=(1,2,3)(4,5,6)\in S_6.\]
The centralizer of $\pi$ is a subgroup  $W\le S_6$ of order $18$ with the following generators:
\[W=\langle (1,2,3), (1, 6, 2, 4, 3, 5)\rangle.\]
Using Sage we find a list $H_1,\ldots, H_{9}$ of representatives of the conjugacy classes of subgroups of $W$. For each subgroup $H\in\{H_1,\ldots, H_9\}$ we then check the conditions
\[H\subseteq\Stab_W(3)\cup\Stab_W(6),\;\;\; H\not\subseteq\Stab_W(3),  \;\;\;H\not\subseteq\Stab_W(6).\]

We find that none of the groups $H_k$ have all of these properties; hence, the output list $\calP$ will necessarily be empty.

\underline{\textit{Case 2:}} $7+4c=0$. Setting $c=-7/4$ we compute the polynomial $\Phi_{3,c}$ and obtain the factorization
\[2^6\cdot\Phi_{3,c}(x)=F(2x)^2,\]
where $F(x)=x^3+x^2-9x-1$. It follows that $\calS_{\Phi_{3,c}}=\calS_F$ and in particular $\delta(\calS_{\Phi_{3,c}})=\delta(\calS_F)$. To show that $\mathcal E^{\ast}(f,3)$ holds, suppose that $\Phi_{3,f}$ (or equivalently $\Phi_{3,c}$) does not have a root in $K$. Then $F$ does not have a root in $K$, and is therefore irreducible. By Corollary \ref{irred_density_cor}, this implies that $\delta(\calS_F)<1$. Thus, we obtain $\delta(\calS_{\Phi_{3,c}})<1$ and therefore $\delta(\calS_{\Phi_{3,f}})<1$. This proves that $\mathcal E^{\ast}(f,3)$ holds. It remains to show that $\mathcal E(f,3)$ holds. Suppose that $f$ (or equivalently $f_c$) has a point of period 3 in $K_{\p}$ for every prime $\p$. This implies that $\Phi_{3,c}$, and therefore $F$, has a root in $K_{\p}$ for every prime $\p$. It follows from Corollary \ref{irred_density_cor} that $F$ is reducible in $K[x]$, so $F(2x)$ has a root $\alpha\in K$. Note that $\alpha$ is a root of $\Phi_{3,c}$, so it has period 1 or 3 under $f_c$. If the period is 1, then $\alpha$ is fixed by $f_c$, so $\alpha$ is a common root of the polynomials $\Phi_{3,c}$ and $f_c(x)-x=x^2-x-7/4$. However, we find that
\[\Res(\Phi_{3,c}(x), x^2-x-7/4)=49\ne 0,\]
so these polynomials have no common root. Hence, $\alpha$ must have period 3. This shows that $f_c$ (and therefore $f$) has a point of period 3 in $K$. Thus, we conclude that $\mathcal E(f,3)$ holds.

\underline{\textit{Case 3:}} $7+4c+16c^2=0$. In this case we have the factorization
\[3^3\cdot 2^6\cdot \Phi_{3,c}(x)=P(2x)^3Q(2x),\] where
\[P(x)=3x-4c+1\;\;\;\text{and}\;\;\; Q(x)=x^3+(4c+1)x^2+(12c-2)x-8c-15.\] Since $P(x)$ is linear, $\Phi_{3,c}$ has a root in $K$, and therefore $\mathcal E^{\ast}(f,3)$ holds. We now show that $\mathcal E(f,3)$ holds. Assuming that $f$ has a point of period 3 in $K_{\p}$ for every prime $\p$, we must show that $f$ has a point of period 3 in $K$. Fix any prime $\p$ of $K$ and let $\alpha\in K_{\p}$ have period 3 under $f_c$. Then $\alpha$ is a root of $\Phi_{3,c}$, so it must be a root of either $P(2x)$ or $Q(2x)$. Note that the polynomial $P(2x)$ has only one root in any extension of $K$, namely the root $(4c-1)/6$. Moreover, a straightforward calculation shows that this root is fixed by $f_c$. Hence, $\alpha$ cannot be a root of $P(2x)$ and must therefore be a root of $Q(2x)$. We have thus shown that $Q$ has a root in $K_{\p}$ for every prime $\p$. By Corollary \ref{irred_density_cor}, this implies that $Q$ is reducible in $K[x]$. Hence, $Q(2x)$ has a root $\beta\in K$. Note that $\beta$ is a root of $\Phi_{3,c}$, so it has period 1 or 3 under $f_c$. We compute
\[\Res(Q(2x), x^2-x+c)=-512c\ne 0,\]
so $\beta$ cannot be a root of $x^2-x+c$, and is therefore not fixed by $f_c$. Hence, the period of $\beta$ must be 3. This proves that $f$ has a point of period 3 in $K$, as required.
\end{proof}

\subsection{The local-global principle for period 4}
We turn now to consider the statement $\mathcal E^{\ast}(f,4)$.

\begin{lem}\label{period4_splitting}
Let $f\in K[x]$ be a quadratic polynomial and let $n=4$. Suppose that $\disc\Phi_{4,f}\ne 0$ and that the statement \ref{lgp} fails to hold. Then the Galois group of $\Phi_{4,f}$ is isomorphic to $\Z/2\Z\times\Z/2\Z$, and $\Phi_{4,f}$ is the product of six irreducible quadratic polynomials in $K[x]$. Moreover, for every quadratic factor $p$ of $\Phi_{4,f}$, the map $f^2$ interchanges the roots of $p$ in $\bar K$.
\end{lem}
\begin{proof}
We apply Theorem \ref{main_alg_thm} with $n=4$ and $s=2$. Carrying out the steps of Algorithm \ref{main_algorithm} with input $(4,2)$, we have
\[N=12,\; r = 3,\;\;\text{and}\;\; \pi=(1,2,3,4)(5,6,7,8)(9,10,11,12)\in S_{12}.\]
The centralizer of $\pi$ is a subgroup $W\le S_{12}$ of order 384 with the following generators:
\[W=\langle(1,7,4,6,3,5,2,8)(9,10,11,12), (1,6,11,2,7,12,3,8,9,4,5,10)\rangle.\]
Using Sage we obtain a list $H_1,\ldots, H_{164}$ of representatives of the conjugacy classes of subgroups of $W$. For each subgroup $H\in\{H_1,\ldots, H_{164}\}$ we then check the conditions
\[H\subseteq \bigcup_{i=1}^3\Stab_W(4i) \;\;\text{and}\;\; H\not\subseteq\Stab_W(4i) \;\;\text{for}\;\; 1\le i\le 3.\]
We find that exactly one of the subgroups $H_m$ has these properties, namely the group
\[H= \langle(1, 3)(2, 4)(5, 7)(6, 8), (5, 7)(6, 8)(9, 11)(10, 12)\rangle\cong \Z/2\Z\times\Z/2\Z.\]
For this group we then compute the set
\[I=\{|H: \Stab_{H}(4i)|:1\le i\le 3\}=\{2\}.\]
The output of the algorithm consists of just the pair $(H,I)$. It follows from Theorem \ref{main_alg_thm} that the Galois group of $\Phi_{4,f}$ is isomorphic to $H$ and hence to $\Z/2\Z\times\Z/2\Z$. Moreover, the theorem implies that $I$ is the set of degrees of irreducible factors of $\Phi_{4,f}$. Since $I=\{2\}$, this shows that every irreducible factor of $\Phi_{4,f}$ has degree 2, and therefore $\Phi_{4,f}$ factors as the product of six irreducible quadratic polynomials. 

It remains to prove the last statement in the lemma. Let $G$ be the Galois group of $\Phi_{4,f}$ and let $R$ be the set of roots of $\Phi_{4,f}$ in $\bar K$. Let $p$ be an irreducible quadratic factor of $\Phi_{4,f}$ with roots $\alpha,\beta\in R$; we must show that $f^2(\alpha)=\beta$. By Theorem \ref{main_alg_thm}, there exist a bijection $u:\{1,\ldots, 12\}\to R$ and an isomorphism $\rho:G\to H$ such that
\begin{equation}\label{period4_bijections_property}
u\circ\pi=f\circ u\;\;\text{and}\;\;g\circ u=u\circ(\rho(g))\;\;\text{for every}\;\;g\in G.
\end{equation}

Note that the group $H$ has the following property: 
\begin{equation}\label{period4_H_property}
\text{For every}\;\; h\in H\;\;\text{and every}\;\;k\in\{1,\ldots, 12\}, \;\;h(k)\in\{k,\pi^2(k)\}.
\end{equation}

Since $\alpha$ and $\beta$ are Galois conjugates, there is an element $g\in G$ such that $\beta=g(\alpha)$. Let $h=\rho(g)$ and $k=u^{-1}(\alpha)$. By \eqref{period4_H_property}, we have either $h(k)=k$ or $h(k)=\pi^2(k)$. If $h(k)=k$, then $u(h(k))=u(k)=\alpha$, so $u\circ h\circ u^{-1}(\alpha)=\alpha$. By \eqref{period4_bijections_property} we know that $u\circ h\circ u^{-1}=g$, so $g(\alpha)=u\circ h\circ u^{-1}(\alpha)=\alpha$. This is a contradiction since $g(\alpha)=\beta\ne\alpha$. Hence, $h(k)$ cannot equal $k$, and so $h(k)=\pi^2(k)$. By similar reasoning as above, this implies that $g(\alpha)=u\circ\pi^2\circ u^{-1}(\alpha)=f^2(\alpha)$. We conclude that $\beta=f^2(\alpha)$, as required.
\end{proof}

The above lemma suggests that in order to understand cases where the statement $\mathcal E^{\ast}(f,4)$ fails, it is important to study the property of $\Phi_{4,f}$ splitting into quadratic factors each of whose roots are interchanged by $f^2$. We now aim to show that this behavior is highly exceptional.

\begin{lem}\label{period4_param}
Let $F$ be a number field and let $c\in F$. Suppose that $\alpha\in F$ has period $4$ under the map $f_c$. Then there exist $u,v\in F$ such that
\begin{equation}\label{period4_curve}
v^2=-u(u^2+1)(u^2-2u-1)
\end{equation}
and the following relations hold:
\begin{align}
\label{period4_cparameter}&c = \frac{(u^2-4u-1)(u^4+u^3+2u^2-u+1)}{4u(u^2-1)^2},\\
\label{period4_conjugates}&\alpha = \frac{u-1}{2(u+1)}+\frac{v}{2u(u-1)}\;\;,\;\;f_c^2(\alpha) = \frac{u-1}{2(u+1)}-\frac{v}{2u(u-1)}.
\end{align}
\end{lem}
\begin{proof}
See  Proposition 3.4 in \cite{jxd}.
\end{proof}

\begin{lem}\label{quad_factor_lem}
Let $c\in K$. Suppose that $\disc\Phi_{4,c}\ne 0$ and that $\Phi_{4,c}$ has a monic irreducible quadratic factor $p\in K[x]$ whose roots in $\bar K$ are interchanged by the map $f_c^2$. Then there is an element $u\in K$ such that \eqref{period4_cparameter} holds and
\begin{equation}\label{quad_factor_param}
p(x)=x^2-\frac{u-1}{u+1}\cdot x+\frac{u^6+u^5-7u^4+2u^3-9u^2-3u-1}{4u(u^2-1)^2}.
\end{equation}
\end{lem}
\begin{proof}
Let $\alpha\in\bar K$ be a root of $p$ and let $F=K(\alpha)$. Since $\alpha$ is a root of $\Phi_{4,c}$ and $\disc\Phi_{4,c}\ne 0$, $\alpha$ has period 4 under $f_c$. By Lemma \ref{period4_param}, there exist $u, v\in F$ such that \eqref{period4_curve} - \eqref{period4_conjugates} hold. Now, since $\alpha$ and $f_c^2(\alpha)$ are the roots of $p$, we have
\[p(x)=x^2-(\alpha+f_c^2(\alpha))\cdot x+\alpha\cdot f_c^2(\alpha).\]
Using the relations \eqref{period4_curve} and \eqref{period4_conjugates} we obtain the expression \eqref{quad_factor_param}. Thus, we have shown that there is an element $u\in F$ such that \eqref{period4_cparameter} and \eqref{quad_factor_param} hold. However, since the coefficients of $p$ belong to $K$, we have in particular that $(u-1)/(u+1)\in K$, which implies that $u\in K$. This completes the proof.
\end{proof}

\begin{prop}\label{period4_finiteness}
There exist at most finitely many elements $c\in K$ such that the polynomial $\Phi_{4,c}$ has more than two monic irreducible quadratic factors, each with the property that its roots are interchanged by $f_c^2$.
\end{prop}
\begin{proof}
A straightforward calculation shows that for every $c\in K$,
\begin{equation}\label{phi4_disc}
\disc \Phi_{4,c} = (5 + 4c)^2(5 - 8c + 16c^2)^3(135 + 108c + 144c^2 + 64c^3)^4.
\end{equation}
It follows in particular that there are only finitely many $c\in K$ for which $\disc\Phi_{4,c}=0$. Hence, it suffices to prove the result assuming that $\disc\Phi_{4,c}\ne 0$. Let $c\in K$. Suppose that $\disc\Phi_{4,c}\ne 0$ and that $\Phi_{4,c}$ has three monic irreducible quadratic factors $p_1,p_2,p_3\in K[x]$, each with the property that its roots are interchanged by $f_c^2$. By Lemma \ref{quad_factor_lem}, there are elements $u_1,u_2,u_3\in K$ such that
\begin{equation}\label{quad_factors_period4}
p_i(x)=x^2-\frac{u_i-1}{u_i+1}\cdot x+\frac{u_i^6+u_i^5-7u_i^4+2u_i^3-9u_i^2-3u_i-1}{4u_i(u_i^2-1)^2}
\end{equation}
and
\begin{equation}\label{c_expressions_period4}
c=\frac{(u_i^2-4u_i-1)(u_i^4+u_i^3+2u_i^2-u_i+1)}{4u_i(u_i^2-1)^2}.
\end{equation}

Since the polynomials $p_i$ are distinct, \eqref{quad_factors_period4} implies that the elements $u_i$ are distinct. Letting $u=u_1$, it follows that at least one of the elements $u_2, u_3$ must lie outside the set $\{u,-1/u\}$. Let $t\in\{u_2,u_3\}$ be different from $u$ and $-1/u$. By \eqref{c_expressions_period4} we have
\[\frac{(u^2-4u-1)(u^4+u^3+2u^2-u+1)}{4u(u^2-1)^2}=\frac{(t^2-4t-1)(t^4+t^3+2t^2-t+1)}{4t(t^2-1)^2}.\]
Clearing denominators and regrouping, this equation becomes
\begin{equation}\label{period4_initial_eq}
(t-u)(ut+1)F(u,t)=0,
\end{equation}
where $F\in\Q[x,y]$ is given by
\[F(x,y)=(x^2 - 1)^2y^4 + (4x)^2y^3 - 2(x^2 - 1)(x^2 - 8x - 1)y^2 - (4x)^2y + (x^2 - 1)^2.\]
Let $C$ be the algebraic curve over $\Q$ defined by the equation $F(x,y)=0$. Since $t$ is different from $u$ and $-1/u$,  \eqref{period4_initial_eq} implies that $F(u,t)=0$, and therefore $(u,t)$ is a $K$-rational point on $C$. Using functionality for computations with algebraic curves in Magma \cite{magma}, we find that $C$ has genus 9.  Hence, by Faltings' theorem, $C$ has only finitely many $K$-rational points. Since $u$ is a coordinate of a $K$-rational point on $C$, there are only finitely many options for $u$. The relation \eqref{c_expressions_period4} therefore implies that there are only finitely many options for $c$.
\end{proof}

\begin{thm}\label{period4_lgp}
There exist at most finitely many elements $c\in K$ for which the statement $\mathcal E^{\ast}(f_c,4)$ is false. The same holds for the statement $\mathcal E(f_c,4)$.
\end{thm}
\begin{proof}
 It suffices to prove the result assuming that $\disc\Phi_{4,c}\ne 0$. Suppose that $c\in K$ is such that $\disc\Phi_{4,c}\ne 0$ and at least one of the statements $\mathcal E^{\ast}(f_c,4)$ and  $\mathcal E(f_c,4)$ fails to hold. By Lemma \ref{stars_lem}, this implies that $\mathcal E^{\ast}(f_c,4)$ does not hold. Applying Lemma \ref{period4_splitting} we conclude that $\Phi_{4,c}$ factors as the product of six irreducible quadratic polynomials, each with the property that its roots are interchanged by the map $f_c^2$. Since $\Phi_{4,c}$ is monic, we may assume without loss of generality that every quadratic factor is monic. Proposition \ref{period4_finiteness} now implies that there are only finitely many possible values for $c$.
\end{proof}

Theorem \ref{period4_lgp} shows that there are, up to linear conjugacy, only finitely many exceptions to the local-global principle $\mathcal E^{\ast}(f,4)$ when $f$ is a quadratic polynomial. In the case $K=\Q$ we can go further and show that there is \textit{no} exception to this principle.

\begin{thm}\label{period4_rational_thm}
Let $f\in\Q[x]$ be a quadratic polynomial. Then $\delta(\calS_{\Phi_{4,f}})<1$. In particular, there exist infinitely many primes $p$ such that $f$ does not have a point of period $4$ in $\Q_p$.
\end{thm}
\begin{proof}
By a theorem of Morton \cite[Thm. 4]{morton_4cycles}, $\Phi_{4,f}$ does not have a root in $\Q$. Hence, the premise of $\mathcal E^{\ast}(f,4)$ is true; we have to show that the conclusion also holds. The proof will be divided into two cases.

\underline{\textit{Case 1:}} $\disc \Phi_{4,f}\ne 0$. Suppose by contradiction that $\delta(\calS_{\Phi_{4,f}})=1$. In other words, the statement $\mathcal E^{\ast}(f,4)$ fails to hold. By Lemma \ref{period4_splitting}, $\Phi_{4,f}$ is then the product of six irreducible quadratic polynomials in $\Q[x]$. However, it is known that $\Phi_{4,f}$ cannot have four or more irreducible quadratic factors; see \cite[Thm. 2.3.5]{panraksa_thesis}. This is a contradiction, so we conclude that $\delta(\calS_{\Phi_{4,f}})<1$, as claimed.

\underline{\textit{Case 2:}} $\disc \Phi_{4,f}=0$. Let $c$ be the unique rational number such that $f$ is linearly conjugate to $f_c$. By \eqref{phi4_disc} we must have $c=-5/4$. Factoring the polynomial $\Phi_{4,c}$ we find that
\[2^{12}\cdot\Phi_{4,c}(x)=P(2x)\cdot Q(2x)^2,\]
where
\[P(x)=x^8 - 4x^7 - 16x^6 + 84x^5 - 6x^4 - 364x^3 + 584x^2 - 836x + 1021\]
and $Q(x)=x^2 + 2x - 1$. It follows from this factorization that $\calS_{\Phi_{4,c}}=\calS_{PQ}$ and therefore $\delta(\calS_{\Phi_{4,f}})=\delta(\calS_{PQ})$. We will use the inequality $\delta(\calS_{PQ})\le\delta(\calS_P)+\delta(\calS_Q)$ to show that $\delta(\calS_{PQ})<1$. The polynomial $Q$ is irreducible, so Corollary \ref{normal_poly_cor} implies that $\delta(\calS_Q)=1/2$. To compute the density of $\calS_P$ we apply Theorem \ref{bb_density_thm}. Using code available in \cite{paper_code}, we obtain $\delta(\calS_P)=7/32$. Therefore,
\[\delta(\calS_{\Phi_{4,f}})=\delta(\calS_{PQ})\le \delta(\calS_P)+\delta(\calS_Q)=7/32+1/2<1.\]
We have shown in both cases that $\delta(\calS_{\Phi_{4,f}})<1$, thus proving the first statement of the theorem. For the second statement, note that there are infinitely many primes not belonging to the set $\calS_{\Phi_{4,f}}$. For every such prime $p$, $\Phi_{4,f}$ does not have a root in $\Q_p$, so part (1) of Proposition \ref{dynatomic_properties} implies that $f$ does not have a point of period $4$ in $\Q_p$.
\end{proof}

\subsection{The local-global principle for period 5}\label{period5_subsection}
We now consider the statement $\mathcal E^{\ast}(f,5)$.
 
\begin{lem}\label{period5_splitting}
Let $f\in K[x]$ be a quadratic polynomial and let $n=5$. Suppose that $\disc\Phi_{5,f}\ne 0$ and that the statement \ref{lgp} fails to hold. Let $G$ be the Galois group of $\Phi_{5,f}$ and let $D$ be the set of degrees of irreducible factors of $\Phi_{5,f}$. Then one of the following must hold:
\begin{itemize}
\item Every 5-cycle of $f$ is $\Gal(\bar K/K)$-invariant, $G\cong\Z/5\Z\times\Z/5\Z$, and $D=\{5\}$;
\item Some 5-cycle of $f$ is $\Gal(\bar K/K)$-invariant, $G\cong S_3\times (\Z/5\Z)$, and $D=\{2,3,5\}$;
\item $G\cong\Z/2\Z\times\Z/2\Z$ and $D=\{2\}$.
\end{itemize}
\end{lem}
\begin{proof}
We apply Theorem \ref{main_alg_thm} with $n=5$ and $s=2$. Carrying out the steps of Algorithm \ref{main_algorithm} with input $(5,2)$, we have
\[N=30,\; r = 6,\;\;\text{and}\;\; \pi=(1,2,3,4,5)\cdots(26, 27, 28, 29, 30)\in S_{30}.\]
The centralizer of $\pi$ is a subgroup $W\le S_{30}$ of order 11250000 generated by the following two permutations:
\begin{align*}
&(1, 18, 8, 12, 3, 20, 10, 14, 5, 17, 7, 11, 2, 19, 9, 13, 4, 16, 6, 15)(21, 26, 23, 28, 25, 30, 22, 27, 24, 29),\\
&(1, 25, 14, 29, 9, 19)(2, 21, 15, 30, 10, 20)(3, 22, 11, 26, 6, 16)(4, 23, 12, 27, 7, 17)(5, 24, 13, 28, 8, 18).
\end{align*}

Using Sage we obtain a list $H_1,\ldots, H_{20844}$ of representatives of the conjugacy classes of subgroups of $W$. For each subgroup $H$ in this list we then check the conditions
\[H\subseteq \bigcup_{i=1}^6\Stab_W(5i) \;\;\text{and}\;\; H\not\subseteq\Stab_W(5i) \;\;\text{for}\;\; 1\le i\le 6.\]
We find that exactly twelve of the groups $H_m$ satisfy these conditions\footnote{The computation of Algorithm \ref{main_algorithm} with input $(5,2)$ was carried out on an iMac with a 2.5 GHz Intel Core i5 processor and 4 GB of memory. The total time required was 8 minutes and 33 seconds.}; we will henceforth denote these groups by $H_1,\ldots, H_{12}$. For each of the groups $H_m$ we compute the elements of the set
\[I_m=\{|H_m: \Stab_{H_m}(5i)|:1\le i\le 6\}.\]

The twelve resulting pairs $(H_m,I_m)$, which form the output of Algorithm \ref{main_algorithm}, are listed in Appendix \ref{appendix}. By Theorem \ref{main_alg_thm}, there is a pair $(H,I)$ in this list such that $G\cong H$ and $D=I$. Moreover, letting $R$ be the set of roots of $\Phi_{5,f}$ in $\bar K$, there exist a bijection $u:\{1,\ldots, 30\}\to R$ and an isomorphism $\rho:G\to H$ such that
\begin{equation}\label{period5_bijections_property}
u\circ\pi=f\circ u\;\;\text{and}\;\;g\circ u=u\circ(\rho(g))\;\;\text{for every}\;\;g\in G.
\end{equation}
We now divide the proof into cases depending on which of the twelve pairs $(H,I)$ is.

\underline{\textit{Case 1:}} $H$ is one of the groups $H_1,\ldots, H_{10}$. Referring to the appendix, we see that $H\cong\Z/5\Z\times\Z/5\Z$ and $I=\{5\}$. Thus, $G\cong\Z/5\Z\times\Z/5\Z$ and $D=\{5\}$. We claim that every $5$-cycle in $R$ is invariant under the action of $\Gal(\bar K/K)$, or equivalently, under the action of $G$. By considering the generators of $H_1,\ldots, H_{10}$ listed in the appendix, one can check that $H$ has the following property:
\begin{equation}\label{period5_H_1_10_property}
\text{For every}\;\;h\in H\;\;\text{and every}\;\;k\in\{1,\ldots, 30\},\;\;h(k)\in\{k, \pi(k), \pi^2(k), \pi^3(k), \pi^4(k)\}.
\end{equation}
Let $\alpha\in R$. We must show that the cycle of $R$ containing $\alpha$ is invariant under $G$. Since every element of $G$ commutes with $f$, it suffices to show that for every $g\in G$, $g(\alpha)=f^i(\alpha)$ for some $i$. Let $g\in G$ and set $h=\rho(g)$ and $k=u^{-1}(\alpha)$. By \eqref{period5_H_1_10_property} we have that $h(k)=\pi^i(k)$ for some $0\le i\le 4$. Thus, $u(h(k))=u(\pi^i(k)))$, so $u\circ h\circ u^{-1}(\alpha)=u\circ\pi^i\circ u^{-1}(\alpha)$ and hence, by \eqref{period5_bijections_property}, $g(\alpha)=f^i(\alpha)$. This proves that every cycle in $R$ is $G$-invariant. Thus, we are in the first case of the lemma.

\underline{\textit{Case 2:}} $H=H_{11}$. Referring to the appendix, we see that $H\cong S_3\times (\Z/5\Z)$ and $I=\{2,3,5\}$. Therefore, $G\cong S_3\times (\Z/5\Z)$ and $D=\{2,3,5\}$. We claim that there is some $5$-cycle in $R$ that is $G$-invariant. Considering the generators of $H_{11}$ given in the appendix, we deduce the following:
\begin{equation}\label{H11_property}
\text{For every}\;\;h\in H\;\;\text{and every}\;\;k\in\{1,2,3,4,5\},\;\;h(k)\in\{k, \pi(k), \pi^2(k), \pi^3(k), \pi^4(k)\}.
\end{equation}
Let $\alpha=u(6)$. By using the above property of $H$ and arguing as in the previous case, one can show that the cycle containing $\alpha$ is invariant under $G$. Thus, we are in the second case of the lemma.

\underline{\textit{Case 3:}} $H=H_{12}$. Referring to the appendix, we see that $H\cong\Z/2\Z\times\Z/2\Z$ and $I=\{2\}$. Therefore, $G\cong\Z/2\Z\times\Z/2\Z$ and $D=\{2\}$. We are thus in the third case of the lemma.
\end{proof}

Lemma \ref{period5_splitting} suggests that in order to understand the possible failure of the local-global principle $\mathcal E^{\ast}(f,5)$, it is important to study the questions of how $f$ could have a Galois-invariant cycle and how the polynomial $\Phi_{5,f}$ could have an irreducible quadratic factor. As we will show below, there are two \textit{dynamical modular curves} that are relevant to these questions. We recall here the definition and basic properties of these curves, and refer the reader to \cite[\S 4.2]{silverman_dynamics} for further details.

Viewing $c$ as an indeterminate, we consider the function field $\Q(c)$ and the polynomial $f_c(x)=x^2+c\in \Q(c)[x]$. Since $f_c$ has coefficients in the field $\Q(c)$, the dynatomic polynomial $\Phi_{n,c}$ also has coefficients in this field. However, one can show that in fact all the coefficients of $\Phi_{n,c}$ lie in the subring $\Z[c]$; hence, $\Phi_{n,c}\in\Z[c][x]$. Let $\Phi_{n}(c,x)$ denote the image of $\Phi_{n,c}(x)$ under the natural isomorphism $\Z[c][x]\to\Z[c,x]$, and let $C_1(n)\subset\A^2=\Spec \Q[c,x]$ be the curve defined by the equation $\Phi_{n}(c,x)=0$. The curve $C_1(n)$ has an automorphism $\sigma$ given by $(c,x)\mapsto(c, x^2+c)$; the quotient of $C_1(n)$ by the group $\langle\sigma\rangle$ is a curve, which we denote by $C_0(n)$. In \cite[p. 335]{morton_dynatomic_curves}, Morton defines a polynomial $\tau_n\in\Q[c,t]$ with the property that $C_0(n)$ is birational to the affine plane curve $\tau_n(c,t)=0$. We will henceforth identify $C_0(n)$ with this plane curve. The quotient map $C_1(n)\to C_0(n)$ is then given by $(c,x)\mapsto (c,\theta_n(c,x))$, where
\[\theta_n(c,x)=x+f_c(x)+\cdots+f_c^{n-1}(x).\]

Consider now a number field $K$ and an element $c\in K$. The following properties of the curves $C_1(n)$ and $C_0(n)$ are easily proved from the definitions:  
\begin{itemize}
\item If $\alpha\in K$ has period $n$ under $f_c$, then $(c,\alpha)$ is a $K$-rational point on $C_1(n)$.
\item Suppose that $\alpha\in\bar K$ has period $n$ under $f_c$, and that the cycle $\{\alpha,f_c(\alpha),\ldots,f_c^{n-1}(\alpha)\}$ is invariant (as a set) under the action of $\Gal(\bar K/K)$. Then $(c,\theta_n(c,\alpha))$ is a $K$-rational point on $C_0(n)$.
\end{itemize}

For the purposes of this section, the curve $C_0(5)$ is especially important. As shown in \cite{flynn-poonen-schaefer}, this is a curve of genus 2. The polynomial $\tau_5(c,t)$ defining the curve is computed in \cite[p. 99]{morton_4cycles}:
\begin{multline}\label{tau5_expression}
\tau_5(c,t)=t^6 + t^5 + t^4(11c + 3) + t^3(18c + 11) + \\t^2(19c^2 + 19c + 44) + t(17c^2 -24c + 36) + 9c^3 + 40c^2 + 28c + 32.
\end{multline}

Another curve playing an important role in this section is the curve $\calX\subset\A^3=\Spec\Q[a,b,c]$ defined by the equations $r_1(a,b,c)=r_0(a,b,c)=0$, where $r_1$ and $r_0$ are the polynomials defined in \eqref{X_curve_equations}. Using Magma \cite{magma} we find that the genus of $\calX$ is 11.

\begin{lem}\label{period5_finiteness_lem}
Let $c\in K$ and suppose that the polynomial $\Phi_{5,c}$ has an irreducible quadratic factor in $K[x]$. Then one of the following holds:
\begin{enumerate}
\item There exists $\theta\in K$ such that $(c,\theta)\in C_0(5)(K)$.
\item There exist $a,b\in K$ such that $(a,b,c)\in \calX(K)$.
\end{enumerate}
\end{lem}
\begin{proof}
Let $p\in K[x]$ be an irreducible quadratic factor of $\Phi_{5,c}$ and let $\alpha\in\bar K$ be a root of $p$. Let $L=K(\alpha)$ be the quadratic extension of $K$ generated by $\alpha$. Since $\alpha$ is a root of $\Phi_{5,c}$, then $\Phi_5(c,\alpha)=0$, so $(c,\alpha)$ is a point in $C_1(5)(L)$. Let $(c,\theta)\in C_0(5)(L)$ be the image of $(c,\alpha)$ under the map $C_1(5)\to C_0(5)$. If $\theta\in K$, then $(c,\theta)$ is a $K$-rational point on $C_0(5)$, and we are in the first case of the lemma. Suppose now that $\theta\notin K$ and let $m(t)=t^2-at-b\in K[t]$ be its minimal polynomial. Since $\theta$ is a root of the polynomial $\tau_5(c,t)\in K[t]$, this polynomial must be divisible by $m$. Beginning with the expression \eqref{tau5_expression} and using long division, we find that the remainder when $\tau_5(c,t)$ is divided by $m(t)$ is given by $r_1(a,b,c)\cdot t + r_0(a,b,c)$, where $r_1$ and $r_0$ are the polynomials defining the curve $\calX$. Since this remainder must be the zero polynomial, it follows that $r_1(a,b,c)=r_0(a,b,c)=0$. Hence, $(a,b,c)$ is a $K$-rational point on $\calX$, and we are in the second case of the lemma.
\end{proof}

We can now prove our main result concerning the local-global principle for points of period 5.

\begin{thm}\label{period5_lgp}
There exist at most finitely many elements $c\in K$ for which the statement $\mathcal E^{\ast}(f_c,5)$ is false. The same holds for the statement $\mathcal E(f_c,5)$.
\end{thm}
\begin{proof}
Since there are only finitely many $c\in K$ such that $\disc\Phi_{5,c}=0$, it suffices to prove the result assuming that $\disc\Phi_{5,c}\ne 0$. Suppose that $c\in K$ is such that $\disc\Phi_{5,c}\ne 0$ and at least one of the statements $\mathcal E^{\ast}(f_c,5)$ and  $\mathcal E(f_c,5)$ fails to hold. By Lemma \ref{stars_lem}, this implies that $\mathcal E^{\ast}(f_c,5)$ does not hold. Applying Lemma \ref{period5_splitting} we see that  either $f_c$ has a $\Gal(\bar K/K)$-invariant 5-cycle, or $\Phi_{5,c}$ has an irreducible quadratic factor in $K[x]$. In the former case, $c$ is a coordinate of a $K$-rational point on $C_0(5)$; in the latter case, by Lemma \ref{period5_finiteness_lem}, $c$ is a coordinate of a $K$-rational point on either $C_0(5)$ or $\calX$. Since $C_0(5)$ has genus 2 and $\calX$ has genus 11, it follows from Faltings' Theorem that there are only finitely many possible values for $c$.
\end{proof}

Restricting now to the case $K=\Q$, we can make Theorem \ref{period5_lgp} more precise. As seen in the proof of the theorem, understanding the failure of the statement $\mathcal E^{\ast}(f,5)$ is closely related -- via Lemma \ref{period5_finiteness_lem} -- to the problem of determining all rational points on the curves $C_0(5)$ and $\calX$.

\begin{lem}\label{C05_points_lem}
The set of rational points on $C_0(5)$ is given by
\[C_0(5)(\Q)=\{(-2,-1), (-4/3,-1), (-16/9,-7/3), (-64/9,10/3)\}.\]
\end{lem}
\begin{proof}
In \cite{flynn-poonen-schaefer}, Flynn-Poonen-Schaefer show that $C_0(5)$ is birational to the hyperelliptic curve
\begin{equation}\label{fps_curve_def}
\calC:\;\; y^2=x^6+8x^5+22x^4+22x^3+5x^2+6x+1,
\end{equation}
and they prove that $\calC$ has exactly six rational points; namely the two points at infinity and the affine points $(0,\pm 1)$ and $(-3,\pm 1)$. A birational map $\psi:C_0(5)\dasharrow\calC$ is given by $\psi=\left(\psi_1,\psi_2\right)$, where
\[\psi_1(c,t)=-\frac{3t^2 + 9t + 3c + 10}{4(t + 1)}\;\;\;\text{and}\;\;\;\psi_2(c,t)=-\frac{R(c,t)}{32(t + 1)^3}.\]

Here,
\begin{multline*}
R(c,t) = 12t^6 + 48t^5 + (24c + 65)t^4 + (72c - 6)t^3 + \\(12c^2 + 122c - 23)t^2 +(24c^2 + 130c + 72)t + 21c^2 + 80c + 76.
\end{multline*}

To determine all rational points on $C_0(5)$, we begin by finding those points $(c_0,t_0)$ having $t_0=-1$. Substituting $t_0=-1$ into the defining equation of $C_0(5)$, we obtain 
\[\tau_5(c,-1)=(2 + c) (4 + 3c)^2.\]
It follows that $(-2,-1)$ and $(-4/3,-1)$ are rational points on $C_0(5)$, and that these are the only points with $t_0=-1$.

Now suppose that $P=(c_0,t_0)$ is a rational point on $C_0(5)$ with $t_0\ne -1$. Then the map $\psi$ is defined at $P$, and $\psi(P)$ is an affine rational point on $\calC$; hence $P$ belongs to the fiber above one of the points $(0,\pm 1), (-3,\pm 1)$. Pulling back each one of these points via the map $\psi$, we find the new points $(-16/9, -7/3)$ and $(-64/9, 10/3)$, and no other point. Hence, these are the only rational points on $C_0(5)$ having $t_0\ne -1$.
\end{proof}

Having determined all rational points on $C_0(5)$, we would like to do the same for the curve $\calX$; unfortunately, the high genus of $\calX$ makes this considerably harder. An extensive search for rational points on $\calX$ yields only the two points $(1, 8, -2)$ and $(-2,-1,-4/3)$, so we expect that there is no other rational point. However, we have not found a proof of this by elementary means; in particular, it is not clear whether $\calX$ covers any curve of lower (positive) genus. It is likely that Chabauty-Coleman techniques can be used to prove that there are only two rational points on $\calX$. Indeed, let $\tilde\calX$ be the nonsingular projective model of $\calX$. Using the construction of $\calX$, one can show that the Jacobian variety of the curve $\calC$ defined in \eqref{fps_curve_def} is an isogeny factor of the Jacobian of $\tilde\calX$. As proved in \cite{flynn-poonen-schaefer}, the group of rational points on $\Jac(\calC)$ has rank 1; this suggests that a Chabauty argument on $\tilde\calX$ should be feasible. The details of this computation will appear in \cite{doyle-krumm-wetherell}.

Returning to the local-global principle, we now work towards improving Theorem \ref{period5_lgp} in the case $K=\Q$. More precisely, our goal is to describe all rational values of $c$ for which the statement $\mathcal E^{\ast}(f_c,5)$ may not hold. As remarked earlier, a relevant question for this purpose is whether the polynomial $\Phi_{5,c}$ can have an irreducible quadratic factor. The following proposition shows that, if we have found all rational points on the curve $\calX$, then this cannot occur.

\begin{prop}\label{period_5_nonsplitting}
Let $f\in\Q[x]$ be a quadratic polynomial and let $c$ be the unique rational number such that $f$ is linearly conjugate to $f_c$. Suppose that the dynatomic polynomial $\Phi_{5,f}$ has an irreducible quadratic factor. Then $c\notin\{-2,-4/3\}$, and there exist $a,b\in\Q$ such that $(a,b,c)\in\calX(\Q)$. In particular, $\calX$ has more than two rational points.
\end{prop}
\begin{proof}
By Lemma \ref{period5_finiteness_lem} we know that either $(c,\theta)\in C_0(5)(\Q)$ for some rational number $\theta$, or $(a,b,c)\in\calX(\Q)$ for some $a,b\in\Q$. If the former case occurs, then Lemma \ref{C05_points_lem} implies that $c=-2, -4/3, -16/9,$ or $-64/9$. Factoring the polynomials $\Phi_{5,c}$ for these four values of $c$, we obtain the following sets $D$ of degrees of irreducible factors:
\begin{center}
\begin{table}[h]\label{phi5_factorization_table}
\begin{tabular}{| c | c |}
\hline
$c$ & $D$ \\
\hline \hline
-2 & \{5,10,15\} \\
\hline
-4/3 & \{10,20\} \\
\hline
-16/9 & \{5,25\} \\
\hline
-64/9 & \{5,25\} \\
\hline
\end{tabular}

\vspace{3mm}
\caption{Degrees of irreducible factors of $\Phi_{5,c}$.}
\end{table}
\end{center}

\vspace{-12mm}
This leads to a contradiction, since there is no quadratic factor of $\Phi_{5,c}$ for these values of $c$. Hence, there must exist $a,b\in\Q$ such that $(a,b,c)\in\calX(\Q)$. Moreover, the above table shows that $c$ cannot be $-2$ or $-4/3$. The last statement of the proposition follows from the fact that the two known rational points on $\calX$ have $c$-coordinate $-2$ or $-4/3$.
\end{proof}

We end this section by showing that if indeed $\calX$ has only two rational points, then there is no exception to the local-global principle $\mathcal E^{\ast}(f,5)$ for quadratic polynomials over $\Q$. 

\begin{thm}\label{period5_rational_thm} Suppose that $\#\calX(\Q)=2$, and let $f\in\Q[x]$ be a quadratic polynomial. Then $\delta(\calS_{\Phi_{5,f}})<1$. In particular, there exist infinitely many primes $p$ such that $f$ does not have a point of period $5$ in $\Q_p$.
\end{thm}
\begin{proof}
We proceed by contradiction. Suppose that $\delta(\calS_{\Phi_{5,f}})=1$ and let $c$ be the unique rational number such that $f$ is linearly conjugate to $f_c$. By a straightforward computation we obtain
\[\disc \Phi_{5,c} = -(31 - 36c + 16c^2 + 64c^3 + 256c^4)^4\cdot \Psi(c)^5,\]
where $\Psi(x)\in\Q[x]$ is an irreducible polynomial of degree 11. It follows that $\disc\Phi_{5,c}\ne 0$ and therefore $\disc\Phi_{5,f}\ne 0$. By work of Flynn-Poonen-Schaefer \cite{flynn-poonen-schaefer} we know that $\Phi_{5,f}$ cannot have a rational root; hence, the assumption that $\delta(\calS_{\Phi_{5,f}})=1$ means that the statement $\mathcal E^{\ast}(f,5)$ does not hold. Applying Lemma \ref{period5_splitting} we conclude that either $f$ has a $\Gal(\bar \Q/\Q)$-invariant 5-cycle, or $\Phi_{5,f}$ has an irreducible quadratic factor. Since we are assuming that $\#\calX(\Q)=2$, Proposition \ref{period_5_nonsplitting} implies that the latter case cannot occur. Hence, $f$ must have a $\Gal(\bar \Q/\Q)$-invariant 5-cycle. As shown in \cite[\S 4]{flynn-poonen-schaefer}, this implies that $c\in\{-2, -16/9,-64/9\}$. Referring to Table 1, we see that the set $D$ of degrees of irreducible factors of $\Phi_{5,f}$ must therefore be $\{5,10,15\}$ or $\{5,25\}$. However, by Lemma \ref{period5_splitting}, the set $D$ can only be $\{5\}, \{2,3,5\}$, or $\{2\}$. This is a contradiction, so we conclude that $\delta(\calS_{\Phi_{5,f}})<1$.
\end{proof}

\section{A strong form of a conjecture of Poonen}\label{strong_poonen_section}

In \cite{poonen_prep}, Poonen conjectures that if $f\in\Q[x]$ is a quadratic polynomial and $n>3$, then $f$ does not have a rational point of period $n$. Motivated by our results in \S\ref{quadratic_section}, we propose the following stronger conjecture.

\begin{conj}\label{strong_poonen_conj}
Let $f\in\Q[x]$ be a quadratic polynomial and $n>3$ an integer. Then $\delta(\calS_{\Phi_{n,f}})<1$.
\end{conj}

We have shown in Theorems \ref{period4_rational_thm} and \ref{period5_rational_thm} that the above conjecture holds for $n=4$, and assuming the curve $\calX$ has only two rational points, that it also holds for $n=5$. For arbitrary values of $n$, we have the following result in support of Conjecture \ref{strong_poonen_conj}.

\begin{thm}\label{strong_poonen_evidence}
Let $n$ be a positive integer. There is a thin\footnote{See \cite[\S 9.1]{serre_lectures} for the definition of a thin set in the sense of Serre.} subset $I(n)\subseteq\Q$ with the following property: if $f\in\Q[x]$ is a quadratic polynomial linearly conjugate to $f_c$ with $c\not\in I(n)$, then $\delta(\calS_{\Phi_{n,f}})<1$.
\end{thm}
\begin{proof}
It is known by a theorem of Bousch \cite[Chap. 3, Thm. 1]{bousch} that the polynomial $\Phi_{n}(c,x)$, viewed as an element of $\C[c,x]$, is irreducible. Hence, by the Hilbert Irreducibility Theorem (see \cite[\S 9]{serre_lectures} or \cite[Chap. 3]{serre_topics}), there is a thin set $I(n)\subset\Q$ such that if $c\notin I(n)$, then the polynomial $\Phi_{n,c}(x)\in\Q[x]$ is irreducible. Applying Corollary \ref{irred_density_cor} we see that if $c\not\in I(n)$, then $\delta(\calS_{\Phi_{n,c}})<1$. Thus, if $c\not\in I(n)$ and $f\in\Q[x]$ is a quadratic polynomial linearly conjugate to $f_c$, then $\delta(\calS_{\Phi_{n,f}})<1$.
\end{proof}

\subsection*{Concluding remarks}
In order to facilitate future progress on the dynamical local-global principle studied in this article, it would be desirable to have theoretical arguments explaining the output of Algorithm \ref{main_algorithm} without having to carry out the computation. In particular, why should the output be empty when $s=2$ and $n=1,2,$ or $3$? What explains the properties \eqref{period4_H_property},  \eqref{period5_H_1_10_property}, and \eqref{H11_property}, without which our analysis for $n=4$ and $5$ would not be possible? And why is it that, among the more than 20,000 groups checked in the algorithm with input $(5,2)$, only 12 pass the tests in step 6 of the algorithm? We hope that these questions will serve to motivate future research on Galois groups of dynatomic polynomials.

\subsection*{Acknowledgements} The author would like to thank John Doyle, Xander Faber, and Joseph Wetherell for helpful conversations during the preparation of this article; Imme Arce for help in generating Figure \ref{roots_labeling}; and the anonymous referees for their careful reading of the manuscript and suggested improvements.

\appendix
\section{Output of Algorithm \ref{main_algorithm} with input (5,2)}\label{appendix}
This appendix summarizes the output produced by our implementation of Algorithm \ref{main_algorithm} when the input to the algorithm is the pair $(5,2)$. The output consists of twelve pairs $(H,I)$, where $H$ is a subgroup of the symmetric group $S_{30}$ and $I=\{|H: \Stab_{H}(5i)|:1\le i\le 6\}$. The twelve pairs are labeled $(H_1, I_1),\ldots, (H_{12},I_{12})$ below. For each pair $(H, I)$ we give a set of generators for the group $H$ and we list the elements of the set $I$. In addition, we describe the structure of $H$ by giving a standard group that is isomorphic to it.

\bigskip
\textbf{Pair} $\bm{(H_{1}, I_1).}$
Generators of $H_1$:
\begin{align*}
&(6,7,8,9,10)(11,13,15,12,14)(16,20,19,18,17)(21,22,23,24,25)(26,29,27,30,28),\\
&(1,2,3,4,5)(6,7,8,9,10)(11,12,13,14,15)(16,17,18,19,20)(26,27,28,29,30)
\end{align*}
Index set $I_1$: $\{ 5 \}$

Structure of $H_1$: $\Z/5\Z\times\Z/5\Z$

\bigskip
\textbf{Pair} $\bm{(H_{2}, I_2).}$
Generators of $H_2$:
\begin{align*}
&(1,4,2,5,3)(6,9,7,10,8)(11,12,13,14,15)(16,17,18,19,20)(26,28,30,27,29),\\
&(1,5,4,3,2)(6,7,8,9,10)(11,12,13,14,15)(21,22,23,24,25)(26,29,27,30,28)
\end{align*}
Index set $I_2$: $\{ 5 \}$

Structure of $H_2$: $\Z/5\Z\times\Z/5\Z$

\bigskip
\textbf{Pair} $\bm{(H_{3}, I_3).}$
Generators of $H_3$:
\begin{align*}
&(6,7,8,9,10)(11,12,13,14,15)(16,18,20,17,19)(21,22,23,24,25)(26,27,28,29,30),\\
&(1,3,5,2,4)(6,8,10,7,9)(11,12,13,14,15)(16,17,18,19,20)(26,30,29,28,27)
\end{align*}
Index set $I_3$: $\{ 5 \}$

Structure of $H_3$: $\Z/5\Z\times\Z/5\Z$

\bigskip
\textbf{Pair} $\bm{(H_{4}, I_4).}$
Generators of $H_4$:
\begin{align*}
&(6,7,8,9,10)(11,15,14,13,12)(16,19,17,20,18)(21,23,25,22,24)(26,27,28,29,30),\\
&(1,5,4,3,2)(6,7,8,9,10)(11,12,13,14,15)(16,17,18,19,20)(26,29,27,30,28)
\end{align*}
Index set $I_4$: $\{ 5 \}$

Structure of $H_4$: $\Z/5\Z\times\Z/5\Z$

\bigskip
\textbf{Pair} $\bm{(H_{5}, I_5).}$
Generators of $H_5$:
\begin{align*}
&(6,7,8,9,10)(11,12,13,14,15)(16,20,19,18,17)(21,23,25,22,24)(26,29,27,30,28),\\
&(1,3,5,2,4)(6,8,10,7,9)(11,12,13,14,15)(16,17,18,19,20)(26,30,29,28,27)
\end{align*}
Index set $I_5$: $\{ 5 \}$

Structure of $H_5$: $\Z/5\Z\times\Z/5\Z$

\bigskip
\textbf{Pair} $\bm{(H_{6}, I_6).}$
Generators of $H_6$:
\begin{align*}
&(6,8,10,7,9)(11,15,14,13,12)(16,19,17,20,18)(21,25,24,23,22)(26,29,27,30,28),\\
&(1,5,4,3,2)(6,7,8,9,10)(11,12,13,14,15)(16,17,18,19,20)(26,29,27,30,28)
\end{align*}
Index set $I_6$: $\{ 5 \}$

Structure of $H_6$: $\Z/5\Z\times\Z/5\Z$

\bigskip
\textbf{Pair} $\bm{(H_{7}, I_7).}$
Generators of $H_7$:
\begin{align*}
&(6,8,10,7,9)(11,13,15,12,14)(16,20,19,18,17)(21,22,23,24,25)(26,28,30,27,29),\\
&(1,3,5,2,4)(6,8,10,7,9)(11,12,13,14,15)(16,17,18,19,20)(26,30,29,28,27)
\end{align*}
Index set $I_7$: $\{ 5 \}$

Structure of $H_7$: $\Z/5\Z\times\Z/5\Z$

\bigskip
\textbf{Pair} $\bm{(H_{8}, I_8).}$
Generators of $H_8$:
\begin{align*}
&(6,8,10,7,9)(11,13,15,12,14)(16,19,17,20,18)(21,24,22,25,23)(26,27,28,29,30),\\
&(1,3,5,2,4)(6,8,10,7,9)(11,12,13,14,15)(16,17,18,19,20)(26,30,29,28,27)
\end{align*}
Index set $I_8$: $\{ 5 \}$

Structure of $H_8$: $\Z/5\Z\times\Z/5\Z$

\bigskip
\textbf{Pair} $\bm{(H_{9}, I_9).}$
Generators of $H_9$:
\begin{align*}
&(6,8,10,7,9)(11,15,14,13,12)(16,19,17,20,18)(21,25,24,23,22)(26,29,27,30,28),\\
&(1,3,5,2,4)(6,8,10,7,9)(11,12,13,14,15)(16,17,18,19,20)(26,30,29,28,27)
\end{align*}
Index set $I_9$: $\{ 5 \}$

Structure of $H_9$: $\Z/5\Z\times\Z/5\Z$

\bigskip
\textbf{Pair} $\bm{(H_{10}, I_{10}).}$
Generators of $H_{10}$:
\begin{align*}
&(6,7,8,9,10)(11,13,15,12,14)(16,19,17,20,18)(21,24,22,25,23)(26,28,30,27,29),\\
&(1,5,4,3,2)(6,7,8,9,10)(11,12,13,14,15)(16,17,18,19,20)(26,29,27,30,28)
\end{align*}
Index set $I_{10}$: $\{ 5 \}$

Structure of $H_{10}$: $\Z/5\Z\times\Z/5\Z$

\bigskip
\textbf{Pair} $\bm{(H_{11}, I_{11}).}$
Generators of $H_{11}$:
\begin{align*}
&(16,26,21)(17,27,22)(18,28,23)(19,29,24)(20,30,25),\\
&(1,4,2,5,3)(6,11)(7,12)(8,13)(9,14)(10,15)(21,26)(22,27)(23,28)(24,29)(25,30)
\end{align*}
Index set $I_{11}$: $\{ 2, 3, 5 \}$

Structure of $H_{11}$: $S_3\times\left(\Z/5\Z\right)$

\bigskip
\textbf{Pair} $\bm{(H_{12}, I_{12}).}$
Generators of $H_{12}$:
\begin{align*}
&(11,16)(12,17)(13,18)(14,19)(15,20)(21,26)(22,27)(23,28)(24,29)(25,30),\\
&(1,6)(2,7)(3,8)(4,9)(5,10)(21,26)(22,27)(23,28)(24,29)(25,30)
\end{align*}
Index set $I_{12}$: $\{ 2 \}$

Structure of $H_{12}$: $\Z/2\Z\times\Z/2\Z$

\bibliography{ref_list}
\bibliographystyle{amsplain}
\end{document}